\documentclass[11pt]{amsart}
\usepackage{amsthm}
\usepackage{graphicx,amsmath,amssymb}
\usepackage{color}
\usepackage{enumerate}
cm \hoffset -1.5 true cm \marginparwidth= 2 true cm

\newtheorem{thm}{Theorem}[section]

\newtheorem{lem}[thm]{Lemma}
\newtheorem{prop}[thm]{Proposition}

\newcommand{\supp}{\text{supp }}

\numberwithin{equation}{section}

\newcommand{\al}{\alpha}

\newcommand{\ep}{\varepsilon}

\allowdisplaybreaks[1]

\begin{document}
\title[Almost sure local wellposedness]
{Almost sure local wellposedness of\\ energy critical fractional Schr\"odinger equations\\ with hartree nonlinearity}

\author[G. Hwang]{Gyeongha Hwang}
\address{Department of Mathematical Sciences, Ulsan National Institute of Science and Technology, Ulsan, 689-798, Republic of Korea}
\email{ghhwang@unist.ac.kr}


\begin{abstract}
We consider a Cauchy problem of energy-critical fractional Schr\"odinger equation with Hartree nonlinearity below the energy space. Using a method of randomization of functions on $\mathbb{R}^d$ associated with the Wiener decomposition, introduced by \'{A}. B\'{e}nyi, T. Oh, and O. Pocovnicu \cite{beohpo1,beohpo2}, we prove that the Cauchy problem is almost surely locally well-posed. Our result includes Hartree Schr\"odinger equation ($\alpha = 2$).
\end{abstract}
\maketitle

\section{Introduction}
 In this paper we consider the following Cauchy problem of the fractional nonlinear Schr\"odinger equations:
\begin{align}\label{main eqn}\left\{\begin{array}{l}
i\partial_tu = |\nabla|^\alpha u + F(u)\;\;
\mbox{in}\;\;\mathbb{R}^{1+d},
\\
u(x,0) = \phi(x) \in H^s\;\; \mbox{in}\;\;\mathbb{R}^d,
\end{array}\right.
\end{align}
where $|\nabla| = (-\Delta)^\frac12$, $d \ge 3$, $1 < \alpha \leq 2$, and
$F(u)$ is the nonlinear term of Hartree type given by
$$
F(u) =
\mu (|\cdot|^{-2\al}*|u|^2)u,\;\; \mu \in \mathbb{R}\setminus \{0\}.
$$

Fractional Schr\"odinger equation appears in fractional quantum mechanics (see \cite{la1, la2, la3}), where Laskin generalized the Brownian-like quantum mechanical path, in the Feynman path integral approach to quantum mechanics, to the $\alpha$-stable L\'evy-like quantum mechanical path.

The solution $u$ of \eqref{main eqn} formally satisfies the mass and energy conservation laws:
\begin{align}\begin{aligned}\label{consv}
&\quad m(u) = \|u(t)\|^2_{L^2}, \\  
&E(u) =  K(u) + P(u), 
\end{aligned}\end{align}
where $$K(u) = \frac12 \big\langle u, |\nabla|^\alpha\,
u\big\rangle, \;\;P(u) = \frac14 \big\langle u, \mu (|x|^{-2\al}*|u|^2)u
\big\rangle.$$ Here $\big\langle\;,\big\rangle$ is the complex inner
product in $L^2$. Hence $H^\frac\al2$ is referred to energy space.

The equation \eqref{main eqn} has a scaling invariance. In fact,
if $u$ is a solution of \eqref{main eqn}, then for any $\lambda > 0$ the scaled function $u_\lambda$ given by
\[
u_\lambda (t, x) = \lambda^{-\frac \al2 + \frac d2} u(\lambda^\alpha \,t, \lambda x )
\]
is also a solution. 
Since $\dot{H}^{\frac \al 2}$-norm is preserved under the scaling $u\mapsto u_\lambda$, \eqref{main eqn} is said to be energy-critical if $s = \frac\al2$. It is also said to be super(sub)-critical if $s < \frac\al2\; (s > \frac\al2$, respectively).


By Duhamel's formula, \eqref{main eqn} is written as an
integral equation
\begin{equation}\label{integral}
u = U(t)\phi -i\mu \int_0^t U(t-t')( (|\cdot|^{-2\al}*|u(t')|^2)u(t'))\,dt'.\end{equation} Here we define the linear propagator $U(t)f$ to be the solution to the linear problem $i\partial_t z = |\nabla|^\al z$ with initial datum $f$. Then it is
formally given  by
\begin{align}\label{int eqn}
U (t)f = \mathcal F^{-1} e^{-it|\xi|^\al} \mathcal F f = (2\pi)^{-d}\int_{\mathbb{R}^d} e^{i( x\cdot \xi - t
|\xi|^\alpha)}\widehat{f}(\xi)\,d\xi,
\end{align}
where $\widehat{f} = \mathcal F f$ denotes the Fourier transform of $f$ such that
$\widehat{f}(\xi) = \int_{\mathbb{R}^d} e^{- ix\cdot \xi} f(x)\,dx$ and  we denote its inverse Fourier transform by  $\mathcal F^{-1} g(x) = (2\pi)^{-d}\int_{\mathbb{R}^d} e^{ix\cdot \xi} g(\xi)\,d\xi$.

For the linear propagator $U(t)$, Strichartz estimate is known to hold
\begin{lem}[Theorem 2 in \cite{cox}]\label{str}
Let $d \geq 2$ and $2/q + d/r = d/2,\; 2\le q, r \le \infty$.
\begin{align*}
\||\nabla|^{-\frac{2-\alpha}{q}}U(t)f\|_{L_T^qL^r} \lesssim \|f\|_{L^2_x}.
\end{align*}
The implicit constant does not depend on $T > 0$.
\end{lem}
Here and after $L_T^qX$ denotes mixed normed space $L^q([-T,T]; X(\mathbb R^d))$ for a Banach space $X$ on $\mathbb R^d$ and $L_t^qX L^q(\mathbb R; X)$. Due to the weak dispersion of $U(t)$ the estimate accompanies a derivative loss of order $2$-$\alpha$. But if one imposes radial assumptions or angularly regular condition on $f$, then a derivative loss can be recovered and even a regularity gain can be obtained (see \cite{cholee, guwa}).

Using Lemma \ref{str}, the local well-posedness of \eqref{main eqn} can be shown in the subcritical case ($s > \frac{\al}{2}$). Actually a little revision of Proposition 4.1 in \cite{chho} gives
\begin{prop}\label{subcri-cauchy}
Let $s > \frac \al2$. If $\phi \in H^s$ then there exists a positive time $T$ such that \eqref{main eqn}
has a unique solution $u \in C([-T,T]; H^s) \cap L^2_T(H_{2d/(d-2)}^{s-(2-\al)/2})$. 
\end{prop}

On the other hand, when $s = \frac{\al}{2}$, by using radial Strichartz estimate, the local and small data global well-posedness to \eqref{main eqn} are proven under the radial assumption of $\phi$ as follows.
\begin{prop} [Theorem 5.2 in \cite{chho}] \label{cri-cauchy}
Let $\frac{2d}{2d-1} \le \al \le 2$ and $\phi \in H_{rad}^\frac{\alpha}2$. then there exists a positive time $T$ such that \eqref{main eqn}
has a unique solution $u \in C([-T,T]; H^\frac\al2) \cap L^3_TH_r^\frac{\alpha}2$, $r = \frac{2n}{n-\frac{2\alpha}3}$. If $\|\phi\|_{\dot H^\frac\al2}$ is sufficiently small, then \eqref{main eqn} is globally well-posed.
\end{prop}
\noindent
Recently the author and collaborators of \cite{chs} obtained global well-posedness for $\frac{2d}{2d-1} < \al < 2$ without smallness when $\mu > 0$ and with $\|\phi\|_{\dot H^\frac\al2} < \|W_\al\|_{\dot H^\frac\al2}$ when $\mu < 0$, where $W_\al$ is a steady state solution of \eqref{main eqn}. Also See \cite{gswz} for power type. In \cite{hs} a power type case was treated in some critical regularity without radial assumption. When $\al =2$, the equation is much easier to handle, so there exist numerous well-posed and ill-posedness results (see \cite{caz, mxz1, Haoz, tsu, Nak, mxz2, ccl, mxz}).

In this paper we focus on supercritical case ($s < \frac\al2$). Many dispersive equations are known to be ill-posed in supercritical regime (see \cite{alca, buget, buget1, car, chcot}). For fractional Schr\"odinger equation, we also observe some negative results. One can readily show the following with a slight modification of illposedness in \cite{cgy, hele}. So we omit the proof.
\begin{prop}
If $s < \al(1 - \frac{\al}{2})$ and the flow map $\phi \mapsto u$ exists in a small neighborhood of the origin as a map from $H^s(\mathbb R^d)$ to $C([-T, T]; H^s)$, then it fails to be $C^{\,3}$ at the origin.
\end{prop}
\noindent
Nonetheless, using probabilistic arguments, Bourgain \cite{bour2}, Burq-Tzvetkov \cite{butz1, butz2}, Colliander-Oh \cite{cooh} and B\'enyi-Oh-Pocovinicu \cite{beohpo1, beohpo2} established positive results on subsets of $H^s$ for the supercritical case (see also \cite{thom, deng, napasta, bobu,OhPo,Po,LuMe}). Especially, in \cite{beohpo1,beohpo2}, the authors introduced a randomization for functions in the usual Sobolev space on $\mathbb{R}^d$.

Many of these works are on the dispersive equation with power type nonlinearity. So we concern the Cauchy problem with random initial data of the equation with Hartree nonlinearity. Because of nonlocal nonlinearity, the problem is more complicated. More precisely, we cannot apply H\"older inequality and bilinear Strichartz estimates (Lemma \ref{xsbbi-1} and Lemma \ref{xsbbi-2}) directly. In order to overcome the difficulty, we decompose functions with respect to frequency as in \cite{ccl}. Now we state our main theorem.
\begin{thm}\label{main thm}
Let $\max(\frac {2\al - 1}{4\al -3}\cdot\frac{\al}{2}, \frac{1}{2}) < s < \frac{\al}{2}$ and $\phi \in H^s$. Consider randomization $\phi^\omega$ defined in \eqref{varcd} with a probability space $(\Omega, \mathcal{F}, P)$ satisfying the condition \eqref{prbcd}. Then \eqref{main eqn} is almost surely locally wellposed in the sense that there exists $C, c, \gamma$ and $\sigma = \frac{\al}{2}+$ such that for each $T \ll 1$, there exists a set $\Omega_T \subset \Omega$ with the following properties:
\begin{enumerate}
\item $P(\Omega \backslash \Omega_T) \leq C \exp\big(-\frac{c}{T^\gamma\|\phi\|_{H^s}^2}\big)$
\item For each $\omega \in \Omega_T$, there exists a unique solution $u \in C([0,T];H^s)$ to \eqref{main eqn} with intial data $\phi^\omega$.
\item Duhamel part of the solution is smoother than initial data, i.e
\[u - U(t)\phi^\omega \in C([0,T];H^\sigma).\]
\end{enumerate}
\end{thm}
The rest of the paper is organized as follows: In Section 2, we briefly review randomization adapted to the Wiener decomposition in \cite{beohpo1,beohpo2}. And in section 3, we introduce Bourgain space $X^{s,b}$ and show bilinear Strichartz esimtates. Lastly in section 4, we shall prove Theorem \ref{main eqn}.

\section{Randomization}
We briefly review randomization adapted to the Wiener decomposition in \cite{beohpo1,beohpo2}.
Let $\psi \in \mathcal{S}$ be a function satisfying
\[\supp \psi \subset [-1,1]^d \text{ and } \sum_{n \in \mathbb{Z}^d}\psi(\xi - n) = 1.\]
And we define pseudo-differential operator $\psi(D-n)$ as a Fourier multiplier
\[\psi(D-n)u(x) = (2\pi)^{-d}\int_{\mathbb{R}^d} e^{i x \cdot \xi} \psi(\xi - n) \widehat{u}(\xi)d\xi.\]
Then given a function $f \in L^2(\mathbb{R}^d)$, we have
\[f = \sum_{n \in \mathbb{Z}^d} \psi(D-n)f.\]

Let $\{g_n\}_{n \in \mathbb{Z}^d}$ be a sequence of independent mean zero complex-valued random variables on a probability space $(\Omega, \mathcal{F}, \mathcal{P})$, where the real and imaginary parts of $g_n$ are independent and endowed with probability distribution $\mu^1_n$ and $\mu^2_n$. We assume there exists $c>0$ such that
\begin{align}\label{prbcd}
\Big|\int_{\mathbb{R}} e^{\gamma x}d\mu^j_n\Big| \leq e^{c\gamma^2},
\end{align}
for all $\gamma \in \mathbb{R}, n \in \mathbb{Z}^d, j =1,2$. Thereafter we define Wiener randomization of $f$ by
\begin{align}\label{varcd}
f^\omega := \sum_{n \in \mathbb{Z}^d} g_n(\omega) \psi(D-n)f.
\end{align}

We recall several well-known useful probabilistic estimates.
\begin{lem}[Lemma 3.1 in \cite{butz1}]\label{prbest-1}
For given $\{c_n\} \in \ell^2(\mathbb{Z}^d)$ and $p \geq 2$, there exists $C>0$ such that
\[\|\sum_{n \in \mathbb{Z}^d} g_n(\omega)c_n\|_{L^p(\Omega)} \leq C\sqrt{q} \|c_n\|_{l^2_n(\mathbb{Z}^d)}.\]
\end{lem}
\begin{lem}[Lemma 2.2 in \cite{beohpo1}]\label{prbest-2}
Given $f \in H^s(\mathbb{R}^d)$, we have for any $\lambda >0$,
\[P\Big(\|f^\omega\|_{H^s(\mathbb{R}^d)} > \lambda \Big) \leq Ce^{-c\lambda^2\|f\|^{-2}_{H^s}}.\]
\end{lem}
\begin{lem}[Lemma 2.3 in \cite{beohpo1}]\label{prbest-3}
Given $f \in L^2(\mathbb{R}^d)$ and finite $p \geq 2$, there exists $C,c >0$ such that for any $\lambda > 0$,
\[P\Big(\|f^\omega\|_{L^p(\mathbb{R}^d)} > \lambda \Big) \leq Ce^{-c\lambda^2\|f\|^{-2}_{L^2}}.\]
In particular, $f^\omega$ is in $L^p(\mathbb{R}^d)$ almost surely.
\end{lem}
Exactly same arguments for Schr\"odinger equation in \cite{beohpo1,beohpo2} give probablistic Strichartz estimates for fractional Schr\"odinger equation. Actually the only property of linear propagator used in those papers is that $L^2$-norm of linear propagator is conserved in time (see Proposition 1.3 in \cite{beohpo1}).
\begin{prop}\label{prbStr-1}
Given $f \in L^2(\mathbb{R}^d)$, let $f^\omega$ be its randomization. Then, given $2 \leq q,r \leq \infty$, for all $T>0$ and $\lambda>0$ there exists $C,c > 0$ such that
\begin{align}
P(\|U(t)f^\omega\|_{L^q_tL^r_x([0,T] \times \mathbb{R}^d)} > \lambda) \leq C\exp\Big(-c\frac {\lambda^2}{T^{\frac 2q}\|f\|^2_{L^2}}\Big).
\end{align}
\end{prop}

\section{Bourgain Space}
We introduce Bourgain space $X^{s,b}$ defined as follows: for $s, b \in \mathbb R$
\[X^{s,b} = \big\{\varphi \in \mathcal S' : \|\varphi\|_{X^{s,b}} := \|\langle \xi \rangle^s \langle \tau - |\xi|^\al \rangle^b \widetilde \varphi (\tau, \xi)
\|_{L^2(\mathbb R \times \mathbb{R}^d)} < \infty\big\}, \]
where $\langle a \rangle = 1  + |a|$ and $\widetilde \varphi$ denotes the time-space Fourier transform. In what follows we mention a few of well-known properties of $X^{s,b}$ space.
\begin{lem}\label{xsb}
Let $T \in (0,1)$ and $b \in \big(\frac 12, \frac 32\big]$. Then for $s \in \mathbb{R}$ and $\theta \in \big[0, \frac 32 - b\big)$ the following hold
\begin{align*}
\|\eta_T(t)U(t)f\|_{X^{s,b}(\mathbb{R} \times \mathbb{R}^d)} &\lesssim T^{\frac 12 - b}\|f\|_{H^s(\mathbb{R}^d)},\\
\|\eta_T(t)\int^t_0 U(t-t')\eta_T(t')F(t')dt'\|_{X^{s,b}(\mathbb{R} \times \mathbb{R}^d)} &\lesssim T^\theta \|F\|_{X^{s,b-1+\theta}(\mathbb{R} \times \mathbb{R}^d)}.
\end{align*}
\end{lem}
\begin{lem}\label{str-xsb}
Let $(q,r)$ satisfy $\frac 2 q + \frac dr = \frac d2,$ and $(d,q,r) \neq (2,2,\infty)$.
Then for $b>\frac 12$ we have
\begin{align*}
\|u\|_{L^{q}_tL^{r}_x(\mathbb{R} \times \mathbb{R}^d)} \lesssim \|u\|_{X^{\frac{2-\al}{q},b}(\mathbb{R} \times \mathbb{R}^d)}.
\end{align*}
\end{lem}
The above lemma follows from Strichartz estimates (Lemma \ref{str}). By interpolation with trivial estimate $\|u\|_{L^2_{t,x}} \lesssim \|u\|_{X^{0,0}}$, we have the following lemma.
\begin{lem}\label{mass-xsb}
Let $ q \ge 2$. Then for $b>\frac 12$ we have
\begin{align*}
\|u\|_{L^q_tL^2_x} \lesssim \|u\|_{X^{0,b(1-\frac 2q)}(\mathbb{R} \times \mathbb{R}^d)}.
\end{align*}
\end{lem}
Because of scaling symmetry, Strichartz estimate is optimal. But if one considers interaction of two different frequency localized data, it is possible to obtain bilinear Strichartz estimate.
\begin{lem}[Lemma 2.2 in \cite{chkl1}]\label{bistr}
Let $d \geq 2$. Suppose that $\supp \widehat{f} \subset A(N_1)$ and $\supp \widehat{g} \subset A(N_2)$ with $N_1 \leq N_2$. Then we have
\begin{align*}
\|U(t)fU(t)g\|_{L_{t,x}^2} &\lesssim \big( \frac {N_1}{N_2} \big)^{\frac {d + \al-2}{4}} (N_1N_2)^{\frac {d-\al}4} \|f\|_{L^2_x}\|g\|_{L^2_x}
=N_1^{\frac{d-1}{2}}N_2^{\frac{1-\al}{2}} \|f\|_{L^2_x}\|g\|_{L^2_x}.
\end{align*}
\end{lem}
\noindent
Moreover, we prove bilinear estimates for data whose Fourier support in a small ball.
\begin{lem}\label{bistrball}
Let $d \geq 2$. Suppose that $\supp \widehat{f} \subset B(\xi_0, \rho_1), $ with $\rho_1, |\xi_0| \ll 1$ and $\supp \widehat{g} \subset A(1)$. Then we have
\begin{align*}
\|U(t)fU(t)g\|_{L_{t,x}^2} &\lesssim \rho_1^{\frac {d-1}{2}} \|f\|_{L^2_x}\|g\|_{L^2_x}.
\end{align*}
\end{lem}

\begin{proof}
By decomposing the Fourier support of $g$ into finite number of sets, rotation and mild dilation, it suffices to prove the estimates when $\supp \widehat{g} \subset B(e_1, \delta)$ for some $0 < \delta \ll 1$. By definition of $U(t)$, we have
\[U(t)f(x)U(t)g(x) = (2\pi)^{-2d}\int e^{i(x\cdot(\xi + \eta) -t(|\xi|^\al + |\eta|^\al))}\widehat{f}(\xi)\widehat{g}(\eta)d\xi d\eta.\]
For each $\bar{\xi} = (\xi_2,\cdots,\xi_d)$, we define a bilinear operator
\[B_{\bar{\xi}}(f,g) = \int_{\mathbb{R}^{1+d}} e^{i(x\cdot(\xi + \eta) -t(|\xi|^\al + |\eta|^\al))}\widehat{f}(\xi_1, \bar{\xi})\widehat{g}(\eta)d\xi_1d\eta.\]
We make the change of variable $\zeta = (\zeta_1, \zeta_2, \cdots, \zeta_{d+1}) = (\xi + \eta, |\xi|^\al + |\eta|^\al)$ with the observation $\big|\frac {\partial \zeta}{\partial (\xi_1, \eta)}\big| = \al|\xi_1|\xi|^{\al -2} - \eta_1|\eta|^{\al - 2}| \sim 1$. Then applying Plancherel's theorem and reversing the change variables $(\zeta \rightarrow (\xi_1, \eta))$, we get
\[\|B_{\bar{\xi}}(f,g)\|_{L_t^2L_x^2} \lesssim \|\widehat{f}(\xi_1, \bar{\xi}) \widehat{g}(\eta) \|_{L^2_{\xi_1, \eta}}.\]
Hence by Mikowski's inequality, we have
\[\|U(t)fU(t)g\|_{L_t^2L_x^2} = \|\int B_{\bar{\xi}}(f,g) d{\bar{\xi}}\|_{L_t^2L_x^2} \lesssim \int \|\widehat{f}(\xi_1, \bar{\xi}) \widehat{g}(\eta) \|_{L^2_{\xi_1, \eta}} d{\bar{\xi}} \lesssim \rho_1^{\frac {d-1}{2}} \|f\|_{L^2_x}\|g\|_{L^2_x}.\]
The last inequality follows from the fact that Fourier support of $f$ is in $B(\xi_0, \rho_1)$.
\end{proof}
From Lemma \ref{bistr}, Lemma \ref{bistrball} and definition of $X^{s,b}$ space, one can prove the following lemma.
\begin{lem}\label{xsbbi-1}
Let $d \geq 2$. Consider $u,v \in X^{0,b}$ for $b > \frac 12$.
Then we have
\begin{enumerate}
\item If $\supp \widehat{u} \subset A(N_1)$ and $\supp \widehat{v} \subset A(N_2)$ with $N_1 \leq N_2$, then
\begin{align*}
\|uv\|_{L_{t,x}^2} \lesssim N_1^{\frac{d-1}{2}}N_2^{\frac{1-\al}{2}} \|u\|_{X^{0,b}}\|v\|_{X^{0,b}}.
\end{align*}
\item If $\supp \widehat{u} \subset B(\xi_0, N)$ and $\supp \widehat{v} \subset A(N_2)$  with $N, |\xi_0| \ll N_2$, then
\begin{align*}
\|uv\|_{L_{t,x}^2} \lesssim N^{\frac{d-1}{2}}N_2^{\frac{1-\al}{2}} \|u\|_{X^{0,b}}\|v\|_{X^{0,b}}.
\end{align*}
\end{enumerate}
\end{lem}
\noindent Furthermore interpolation with trivial inequality
$\|uv\|_{L_{t,x}^2} \lesssim \|u\|_{L^\infty_{t,x}}\|v\|_{L^2_{t,x}} \lesssim \|u\|_{X^{\frac{d}{2}+,\frac{1}{2}+}}\|v\|_{X^{0,0}}$
yields the following useful lemma.
\begin{lem}\label{xsbbi-2}
Let $d \geq 2$. Then, for given small $\ep>0$, we have
\begin{enumerate}
\item If $\supp \widehat{u} \subset A(N_1)$ and $\supp \widehat{v} \subset A(N_2)$ with $N_1 \leq N_2$, then
\begin{align*}
\|uv\|_{L_{t,x}^2} \lesssim N_1^{\frac{d-1}{2}+2\ep}N_2^{\frac{1-\al}{2}+2\ep} \|u\|_{X^{0,\frac12 + \ep}}\|v\|_{X^{0,\frac12 - \ep}}.
\end{align*}
\item If $\supp \widehat{u} \subset B(\xi_0, N)$ and $\supp \widehat{v} \subset A(N_2)$  with $|\xi_0|\sim N_1$ and $N, N_1 \ll N_2$, then
\begin{align*}
\|uv\|_{L_{t,x}^2} \lesssim N^{\frac{d-1}{2}-\ep_1}N_1^{\ep_2}N_2^{\frac{1-\al}{2}+2\ep} \|u\|_{X^{0,\frac12 + \ep}}\|v\|_{X^{0,\frac12 - \ep}}
\end{align*}
where $\ep_1 = \frac{2(d-1)\ep}{1+2\ep}$ and $\ep_2 = \frac{2(d+2\ep)\ep}{1+2\ep}$ so that $\ep_2 - \ep_1 = 2\ep$.
\end{enumerate}
\end{lem}

\section{Almost sure local wellposedness}
We will prove Theorem \ref{main thm}. Given $\phi \in H^s(\mathbb{R}^d)$, let $\phi^\omega$ be its randomization. We concern \eqref{main eqn} with $\phi^\omega$. Let $z(t) := U(t)\phi^\omega$ and $v(t) := u(t) - U(t)\phi^\omega$. Then \eqref{main eqn} becomes
\begin{align}\label{main random eqn}\left\{\begin{array}{l}
i\partial_tv = |\nabla|^\alpha v + F(v + z),\;\;
\mbox{in}\;\;\mathbb{R}
\times \mathbb{R}^d,\\
v(0, x) = 0.
\end{array}\right.
\end{align}
By Duhamel's principle, \eqref{main random eqn} is written as integral equation
\[v(t) = \int_0^t U(t-t')F(v+z)(t')dt'.\]
Let $\eta$ be a smooth cutoff function supported on $[-2,2], \eta = 1$ on $[-1,1]$, and let $\eta_T(t) = \eta(t/T)$. Then we have
\[v(t) = \eta_T(t)\int_0^t U(t-t')\eta_T(t')F(\eta_T(t')v+\eta_T(t')z)(t')dt'.\]
So we define $\mathcal{D}$ by
\[\mathcal{D}v(t) = \eta_T(t)\int_0^t U(t-t')\eta_T(t')F(\eta_T(t')v+\eta_T(t')z)(t')dt'.\]
Now it suffices to prove $\mathcal{D}$ has a fixed point in closed subset of $C_tH^s_x([0,T] \times \mathbb{R}^d)$ outside a set of probability $\leq C\exp\big(-\frac{c}{T^\gamma\|\phi\|_{H^s}^2}\big)$. For that purpose, we show contraction inequality (Proposition \ref{cont}) for $\mathcal{D}$. Then exactly same arguments in p.11 of \cite{beohpo1} give Theorem \ref{main eqn}.
\begin{prop}\label{cont}
Given $\phi \in H^s(\mathbb{R}^d)$, let $\phi^\omega$ be its randomization. Then, there exits $\sigma = 1+$, $b = \frac 12+$ and $\theta = 0+$ such that for each small $T \ll 1$ and $R>0$, we have
\begin{align*}
\|Dv\|_{X^{\sigma,b}} &\leq C_1T^{\theta} \big( \|v\|_{X^{\sigma,b}}^3 + R^3\big)\\
\|Dv - Dw\|_{X^{\sigma,b}} &\leq C_2T^{\theta} \big( \|v\|_{X^{\sigma,b}}^2 +\|w\|_{X^{\sigma,b}}^2 + R^2\big)\|v - w\|_{X^{\sigma, b}},
\end{align*}
outside a set of probability at most $C\exp\Big(-c\frac{R^2}{\|\phi\|^2_{H^s}}\Big)$.
\end{prop}
\begin{proof}
We shall only show first estimate, then second estimate can be also proven similarly. By using Lemma \ref{xsb} and duality, we get
\begin{align*}
\|Dv(t)\|_{X^{\sigma, b}} &\lesssim T^\theta \|F(\eta_Tv + \eta_Tz)\|_{X^{\sigma, b-1+\theta}}\\
&= T^\theta \sup_{\|v_4\|_{X^{0, 1-b-\theta}}\leq 1} \Big|\int\int_{\mathbb{R} \times \mathbb{R}^d} \langle \nabla \rangle^\sigma [F(\eta_Tv + \eta_Tz)]v_4 dxdt \Big|.
\end{align*}
So there exist 6 terms to be considered
\begin{flalign*}
&(1)\ \Big|\int\int_{\mathbb{R} \times \mathbb{R}^d} \langle \nabla \rangle^\sigma (|x|^{-2\al}*|\eta_Tv|^2)\eta_Tvv_4 dxdt\Big|&&\\
&(2)\ \Big|\int\int_{\mathbb{R} \times \mathbb{R}^d} \langle \nabla \rangle^\sigma (|x|^{-2\al}*|\eta_Tz|^2)\eta_Tzv_4 dxdt\Big|&&\\
&(3)\ \Big|\int\int_{\mathbb{R} \times \mathbb{R}^d} \langle \nabla \rangle^\sigma (|x|^{-2\al}*(\eta_Tv \overline{\eta_Tz}+\overline{\eta_Tv} \eta_Tz))\eta_Tvv_4 dxdt\Big|&&\\
&(4)\ \Big|\int\int_{\mathbb{R} \times \mathbb{R}^d} \langle \nabla \rangle^\sigma (|x|^{-2\al}*|\eta_Tv|^2)\eta_Tzv_4 dxdt\Big|&&\\
&(5)\ \Big|\int\int_{\mathbb{R} \times \mathbb{R}^d} \langle \nabla \rangle^\sigma (|x|^{-2\al}*(\eta_Tv \overline{\eta_Tz}+\overline{\eta_Tv} \eta_Tz))\eta_Tzv_4 dxdt\Big|&&\\
&(6)\ \Big|\int\int_{\mathbb{R} \times \mathbb{R}^d} \langle \nabla \rangle^\sigma (|x|^{-2\al}*|\eta_Tz|^2)\eta_Tvv_4 dxdt\Big|.&&
\end{flalign*}

We will estimate each term by using Stricharz estimates, Bilinear Strichartz estimates and probabilistic estimates.\\\\
1st Term : $vvv$ term
\begin{align}\label{term1}
\Big|\int\int_{\mathbb{R} \times \mathbb{R}^d} \langle \nabla \rangle^\sigma ((|x|^{-2\al}*|\eta_Tv|^2)\eta_Tv)v_4 dxdt\Big|.
\end{align}

By H\"older inequality, \eqref{term1} is bounded by
\[\|\langle \nabla \rangle^\sigma ((|x|^{-2\al}*(|\eta_Tv|^2)\eta_Tv)\|_{L^{\frac 1{1-\ep}}_tL^{2}_x}\|v_4\|_{L^{\frac {1}{\ep}}_tL^2_x}.\]
for some small positive $\ep$ such that $\ep < \frac {1}{\al}(\sigma - \frac {\al}{2})$.
From Lemma \ref{mass-xsb}, we have
\[\|v_4\|_{L^{\frac 1 \ep}_tL^2_x} \lesssim \|v_4\|_{X^{0,\widetilde{b}(1-2\ep)}} = \|v_4\|_{X^{0,1-b-\theta}}.\]

In order to deal with nonlocal term, we introduce useful lemmas.
\begin{lem}[Lemma A1 $\sim$ Lemma A4 in \cite{ka}]\label{lei}
For any $s \ge 0$ we have
$$
\||\nabla|^s(uv)\|_{L^r} \lesssim \||\nabla|^s
u\|_{L^{r_1}}\|v\|_{L^{q_2}} + \|u\|_{L^{q_1}}\||\nabla|^s v\|_{L^{r_2}},
$$
where $\frac1{r} = \frac1{r_1} + \frac1{q_2} =
\frac1{q_1}+\frac1{r_2},\quad r_i \in (1, \infty),\; q_i \in
(1,\infty],\quad i = 1, 2.$
\end{lem}
\begin{lem}[Lemma 3.2 in \cite{chooz}]\label{frac bound}
For any $0 < \ep_1 < d - 2\al$ we
have
$$
\left\||x|^{-2\al}*(|u|^2)\right\|_{L^\infty} \lesssim
\|u\|_{L^\frac{2d}{d - 2\al - \ep_1}}\|u\|_{L^{\frac{2d}{d - 2\al +\ep_1}}}.
$$
\end{lem}
By using Lemma \ref{lei}, we get
\begin{align*}
\|\langle \nabla \rangle^\sigma ((|x|^{-2\al}*&(|\eta_Tv|^2)\eta_Tv)\|_{L^{\frac 1{1-\ep}}_tL^{2}_x}
\lesssim \||x|^{-2\al}*(|\eta_Tv|^2)\|_{L^{\frac {1}{1-\ep}}_tL^{\infty}_x}\|\langle \nabla \rangle^\sigma \eta_Tv\|_{L^\infty_tL^2_x}\\
&\quad+ \|\langle \nabla \rangle^\sigma(|x|^{-2\al}*|\eta_Tv|^2)\|_{L^{\frac {2}{1-\ep}}_tL^{\frac {2d}{2\al + \al\ep}}_x}\|\eta_Tv\|_{L^{\frac {2}{1-\ep}}_tL^{\frac {2d}{d - 2\al - \al\ep}}_x}.
\end{align*}

Thereafter, from Lemma \ref{str-xsb}, we obtain
\[\|\langle \nabla \rangle^\sigma \eta_Tv\|_{L^\infty_tL^2_x} \lesssim \|v\|_{X^{\sigma, b}}.\]

For $\||x|^{-2\al}*(|\eta_Tv|^2)\|_{L^{\frac {1}{1-\ep}}_tL^{\infty}_x}$, we use Lemma \ref{frac bound} and H\"older inequality
\[\||x|^{-2\al}*(|\eta_Tv|^2)\|_{L^{\frac {1}{1-\ep}}_tL^{\infty}_x} \lesssim \|\eta_Tv\|_{L^{\frac {2}{1-\ep}}_tL^{\frac {2d}{d-2\al-\al\ep}}_x}\|\eta_Tv\|_{L^{\frac {2}{1-\ep}}_tL^{\frac {2d}{d-2\al + \al\ep}}_x}.\]
Then, from Sobolev embedding, we obtain
\begin{align*}
&\|\eta_Tv\|_{L^{\frac {2}{1-\ep}}_tL^{\frac {2d}{d-2\al-\al\ep}}_x} \lesssim \|\langle \nabla \rangle^{\sigma_1}v\|_{L^{\frac {2}{1-\ep}}_tL^{\frac {2d}{d-2(1-\ep)}}_x},\ \|\eta_Tv\|_{L^{\frac {2}{1-\ep}}_tL^{\frac {2d}{d-2\al+\al\ep}}_x} \lesssim \|\langle \nabla \rangle^{\sigma_2}v\|_{L^{\frac {2}{1-\ep}}_tL^{\frac {2d}{d-2(1-\ep)}}_x},
\end{align*}
where $\sigma_1 = \frac {d-2(1-\ep)}{2} - \frac {d-2\al-\al\ep}{2}$ and $\sigma_2 = \frac {d-2(1-\ep)}{2} - \frac {d-2\al + \al\ep}{2}$.
And Lemma \ref{str-xsb} yield
\[\|\langle \nabla \rangle^{\sigma_1}v\|_{L^{\frac {2}{1-\ep}}_tL^{\frac {2d}{d-2(1-\ep)}}_x}, \|\langle \nabla \rangle^{\sigma_2}v\|_{L^{\frac {2}{1-\ep}}_tL^{\frac {2d}{d-2(1-\ep)}}_x} \lesssim \|v\|_{X^{\sigma,b}},\]
because $\ep < \frac {1}{\al}(\sigma - \frac {\al}{2})$ gives $\sigma_2 + (2-\al)\cdot\frac {1-\ep}{2} < \sigma_1 + (2-\al)\cdot\frac {1-\ep}{2} < \sigma$.

For $\|\langle \nabla \rangle^\sigma(|x|^{-2\al}*|\eta_Tv|^2)\|_{L^{\frac {2}{1-\ep}}_tL^{\frac {2d}{2\al + \al\ep}}_x}$, we use fractional integration Theorem
\[\|\langle \nabla \rangle^\sigma(|x|^{-2\al}*|\eta_Tv|^2)\|_{L^{\frac {2}{1-\ep}}_tL^{\frac {2d}{2\al + \al\ep}}_x} \lesssim \|\langle \nabla \rangle^\sigma(|\eta_Tv|^2)\|_{L^{\frac {2}{1-\ep}}_tL^{\frac {2d}{2d - 2\al + \al\ep}}_x}.\]
Then Lemma \ref{lei} and H\"older inequality give
\[\|\langle \nabla \rangle^\sigma(|\eta_Tv|^2)\|_{L^{\frac {2}{1-\ep}}_tL^{\frac {2d}{2d - 2\al + \al\ep}}_x} \lesssim \|\langle \nabla \rangle^\sigma v\|_{L^\infty_tL^2_x} \|v\|_{L^{\frac {2}{1-\ep}}_tL^{\frac {2d}{d - 2\al + \al\ep}}_x}.\]
By using Solobev inequality and Lemma \ref{str-xsb}, we have
\[\|v\|_{L^{\frac {2}{1-\ep}}_tL^{\frac {2d}{d - 2\al + \al\ep}}_x} \lesssim \|\langle \nabla \rangle^{\sigma_2}v\|_{L^{\frac {2}{1-\ep}}_tL^{\frac {2d}{d-2(1-\ep)}}_x} \lesssim \|v\|_{X^{\sigma,b}}.\]

In conclusion, we get $\eqref{term1}$ is bounded by $\|v\|^3_{X^{\sigma,b}}\|v_4\|_{X^{0,1-b-\theta}}$.
\\

In order to handle remaining terms, we make dyadic decomposition and assume Fourier transform of $z_i$, $v_i$ is supported on the set $\{\xi \sim N_i\}$. In dealing with 2nd, 4th and 6th terms, we may assume $N_1 \leq N_2$.
\\\\
2nd Term : $zzz$ term
\begin{align}\label{term2}
\big|\int\int_{\mathbb{R} \times \mathbb{R}^d} \langle \nabla \rangle^\sigma ((|x|^{-2\al}*z_1z_2)z_3)v_4 dxdt\big|.
\end{align}

We consider two cases separately
\begin{enumerate}[\quad i.]
\item $\max(N_1,N_2,N_3) \sim \text{med}(N_1,N_2, N_3)$
\item $\max(N_1,N_2,N_3) \gg \text{med}(N_1,N_2,N_3)$.
\end{enumerate}

Case (2.i) : $\max(N_1,N_2,N_3) \sim \text{med}(N_1,N_2, N_3)$

By H\"older inequality, \eqref{term2} is bounded by
\[\|\langle \nabla \rangle^\sigma ((|x|^{-2\al}*(z_1z_2))z_3)\|_{L^{\frac 1{1-\ep}}_tL^{2}_x}\|v_4\|_{L^{\frac {1}{\ep}}_tL^2_x}\]
for some small positive $\ep$.
From Lemma \ref{mass-xsb}, we have
\[\|v_4\|_{L^{\frac 1 \ep}_tL^2_x} \lesssim \|v_4\|_{X^{0,\widetilde{b}(1-2\ep)}} = \|v_4\|_{X^{0,1-b-\theta}}.\]
And by using Lemma \ref{lei}, we get
\begin{align*}
\|\langle \nabla \rangle^\sigma ((|x|^{-2\al}*&(z_1z_2))z_3)\|_{L^{\frac 1{1-\ep}}_tL^{2}_x}
\lesssim \||x|^{-2\al}*(z_1z_2)\|_{L^{\frac {3}{2-2\ep}}_tL^{\frac {3d}{2\al}}_x}\|\langle \nabla \rangle^\sigma z_3\|_{L^{\frac {3}{1-\ep}}_tL^{\frac {2d}{d - \frac {4\al}{3}}}_x}\\
&\quad+ \|\langle \nabla \rangle^\sigma(|x|^{-2\al}*(z_1z_2))\|_{L^{\frac {3}{2-2\ep}}_tL^{\frac {3d}{2\al}}_x}\|z_3\|_{L^{\frac {3}{1-\ep}}_tL^{\frac {2d}{d - \frac {4\al}{3}}}_x}.
\end{align*}

We first concern term $\||x|^{-2\al}*(z_1z_2)\|_{L^{\frac {3}{2-2\ep}}_tL^{\frac {3d}{2\al}}_x}$. Fractional integration theorem and H\"older inequality yield
\[\||x|^{-2\al}*(z_1z_2)\|_{L^{\frac {3}{2-2\ep}}_tL^{\frac {3d}{2\al}}_x} \lesssim \|z_1z_2\|_{L^{\frac {3}{2-2\ep}}_tL^{\frac {3d}{3d - 4\al}}_x} \lesssim \|z_1\|_{L^{\frac {3}{1-\ep}}_tL^{\frac {2d}{d - \frac {4\al}{3}}}_x}\|z_2\|_{L^{\frac {3}{1-\ep}}_tL^{\frac {2d}{d - \frac {4\al}{3}}}_x}.\]
Then from $\max(N_1,N_2,N_3) \sim \text{med}(N_1,N_2, N_3)$, we obtain
\begin{align*}
&\||x|^{-2\al}*(z_1z_2)\|_{L^{\frac {3}{2-2\ep}}_tL^{\frac {3d}{2\al}}_x}\|\langle \nabla \rangle^\sigma z_3\|_{L^{\frac {3}{1-\ep}}_tL^{\frac {2d}{d - \frac {4\al}{3}}}_x}\\
&\lesssim \|z_1\|_{L^{\frac {3}{1-\ep}}_tL^{\frac {2d}{d - \frac {4\al}{3}}}_x}\|\langle \nabla \rangle^{\frac \sigma2}z_2\|_{L^{\frac {3}{1-\ep}}_tL^{\frac {2d}{d - \frac {4\al}{3}}}_x}\|\langle \nabla \rangle^{\frac \sigma2}z_3\|_{L^{\frac {3}{1-\ep}}_tL^{\frac {2d}{d - \frac {4\al}{3}}}_x}\\
&+\|\langle \nabla \rangle^{\frac \sigma2}z_1\|_{L^{\frac {3}{1-\ep}}_tL^{\frac {2d}{d - \frac {4\al}{3}}}_x}\|\langle \nabla \rangle^{\frac \sigma2}z_2\|_{L^{\frac {3}{1-\ep}}_tL^{\frac {2d}{d - \frac {4\al}{3}}}_x}\|z_3\|_{L^{\frac {3}{1-\ep}}_tL^{\frac {2d}{d - \frac {4\al}{3}}}_x}.
\end{align*}

For $\|\langle \nabla \rangle^\sigma(|x|^{-2\al}*(z_1z_2))\|_{L^{\frac {3}{2-2\ep}}_tL^{\frac {3d}{2\al}}_x}$, we use fractional integration Theorem
\[\|\langle \nabla \rangle^\sigma(|x|^{-2\al}*(z_1z_2))\|_{L^{\frac {3}{2-2\ep}}_tL^{\frac {3d}{2\al}}_x} \lesssim \|\langle \nabla \rangle^\sigma(z_1z_2)\|_{L^{\frac {3}{2-2\ep}}_tL^{\frac {3d}{3d-4\al}}_x}.\]
Then fractional Leibniz rule and H\"older inequality give
\begin{align*}
\|\langle \nabla \rangle^\sigma(z_1z_2)\|_{L^{\frac {3}{2-2\ep}}_tL^{\frac {3d}{3d-4\al}}_x} \lesssim \|\langle \nabla \rangle^\sigma z_1\|_{L^{\frac {3}{1-\ep}}_tL^{\frac {2d}{d - \frac {4\al}{3}}}_x} \|z_2\|_{L^{\frac {3}{1-\ep}}_tL^{\frac {2d}{d - \frac {4\al}{3}}}_x} + \|z_1\|_{L^{\frac {3}{1-\ep}}_tL^{\frac {2d}{d - \frac {4\al}{3}}}_x}\|\langle \nabla \rangle^\sigma z_2\|_{L^{\frac {3}{1-\ep}}_tL^{\frac {2d}{d - \frac {4\al}{3}}}_x}.
\end{align*}
And from $\max(N_1,N_2,N_3) \sim \text{med}(N_1,N_2,N_3)$, we have
\begin{align*}
&\quad \|\langle \nabla \rangle^\sigma(|x|^{-2\al}*(z_1z_2))\|_{L^{\frac {3}{2-2\ep}}_tL^{\frac {3d}{2\al}}_x}\|z_3\|_{L^{\frac {3}{1-\ep}}_tL^{\frac {2d}{d - \frac {4\al}{3}}}_x}\\
&\lesssim \|z_1\|_{L^{\frac {3}{1-\ep}}_tL^{\frac {2d}{d - \frac {4\al}{3}}}_x}\|\langle \nabla \rangle^{\frac \sigma2}z_2\|_{L^{\frac {3}{1-\ep}}_tL^{\frac {2d}{d - \frac {4\al}{3}}}_x}\|\langle \nabla \rangle^{\frac \sigma2}z_3\|_{L^{\frac {3}{1-\ep}}_tL^{\frac {2d}{d - \frac {4\al}{3}}}_x}\\
&+\|\langle \nabla \rangle^{\frac \sigma2}z_1\|_{L^{\frac {3}{1-\ep}}_tL^{\frac {2d}{d - \frac {4\al}{3}}}_x}\|\langle \nabla \rangle^{\frac \sigma2}z_2\|_{L^{\frac {3}{1-\ep}}_tL^{\frac {2d}{d - \frac {4\al}{3}}}_x}\|z_3\|_{L^{\frac {3}{1-\ep}}_tL^{\frac {2d}{d - \frac {4\al}{3}}}_x}.
\end{align*}

Therefore, from Proposition \ref{prbStr-1} and $\frac{\sigma}{2} < s$, we conclude that
\[\sum_{\max(N_1,N_2,N_3) \sim \text{med}(N_1,N_2, N_3)} \eqref{term2} \lesssim R^3 \|v_4\|_{X^{0,1-b-\theta}}\]
outside a set of probability
\[\leq C\exp{\big(-c\frac {R^2}{T^{\frac {2-2\ep}{3}}\|\phi\|^2_{H^s}}\big)}.\]

Case (2.ii) : $\max(N_1,N_2,N_3) \gg \text{med}(N_1,N_2, N_3)$

Since the case of $\max(N_1,N_2,N_3) \sim N_2$ can be similarly handled, we only deal with the case of $\max(N_1,N_2,N_3) \sim N_3$. Then we consider 4 cases separately.
\begin{enumerate}[\qquad a.]
\item $N_4 \sim N_3 \gg N_3^{\frac {\al - 1}{2\al -1}} \gg N_2 \gg N_1$
\item $N_4 \sim N_3 \gg N_3^{\frac {\al - 1}{2\al -1}} \gg N_2 \sim N_1$
\item $N_4 \sim N_3 \gg N_2 \gg N_3^{\frac {\al - 1}{2\al -1}} \gg N_1$
\item $N_4 \sim N_3 \gg N_2, N_1 \gg N_3^{\frac {\al - 1}{2\al -1}}$.
\end{enumerate}

Subcase (2.ii.a) : $N_4 \sim N_3 \gg N_3^{\frac {\al - 1}{2\al -1}} \gg N_2 \gg N_1$

The spatial Fourier support of $z_1z_2$ is contained in $A(2N_2)$. So we note that $|\nabla|^{2\al - d} \sim N_2^{2\al -d}$ on the spatial Fourier support of $z_1z_2$. Then we have
\begin{align*}
\frac {|\nabla|^{2\al - d}}{N_2^{\al - d}} (z_1z_2) = (2\pi)^{-d} \int\int e^{ i x\cdot(\xi + \eta)}
\chi\Big(\frac{\xi}{N_2}\Big)\Big(\frac{|\xi + \eta|}{N_2}\Big)^{2\al - d}\chi\Big(\frac{\eta}{N_2}\Big) \widehat{z_1}(\xi) \widehat{z_2}(\eta) d\xi d\eta,
\end{align*}
where $\chi$ is supported in $B(0,1)$.
Now we take Fourier series expansion for $\Psi(\xi, \eta) = \chi(\xi)|\xi + \eta|^{2\al - d}\chi(\eta)$ on the cube of side length $2\pi$ which contains the support of $\Psi$ to get
\[\chi(\xi)|\xi + \eta|^{2\al - d}\chi(\eta) = \sum_{k,l \in \mathbb{Z}^d} C_{k,l} e^{ i(k\cdot\xi + l\cdot\eta)}\]
with $\sum_{k,l}|C_{k,l}| \leq C$. Then we have the identity \[\frac {|\nabla|^{\al - d}}{N_1^{\al - d}} (z_1z_2) = \sum_{k,l \in \mathbb{Z}^d} C_{k,l}z_1^kz_2^l,\]
where
$
z_1^k = (2\pi)^{-d}\int e^{ix\cdot \xi}e^{i k\cdot\xi}\widehat{z_1}(\xi)d\xi
\text{ and }z_2^l = (2\pi)^{-d}\int e^{i x\cdot \eta}e^{i l\cdot\eta}\widehat{z_2}(\eta)d\eta.
$

So we need to estimate
\begin{align*}
\sum_{k,l \in \mathbb{Z}^d} C_{k,l} N_2^{2\al - d} \int\int_{\mathbb{R} \times \mathbb{R}^d} \Big|\langle \nabla \rangle^\sigma  (z^k_1z^l_2z_3)v_4\Big|dxdt.
\end{align*}
And since $|\xi + \eta|^\al \lesssim |\xi|^\al + |\eta|^\al$, it suffices to deal with
\begin{align*}
N_2^{2\al - d} \int\int_{\mathbb{R} \times \mathbb{R}^d} &\big|  \langle \nabla \rangle^\sigma z^k_{1}z^l_{2}z_3v_4\big|dxdt,\ N_2^{2\al - d} \int\int_{\mathbb{R} \times \mathbb{R}^d} \big|z^k_{1}\langle \nabla \rangle^\sigma z^l_{2}z_3v_4\big|dxdt\\
&\text{ and } N_2^{2\al - d} \int\int_{\mathbb{R} \times \mathbb{R}^d} \big|z^k_{1}z^l_{2}\langle \nabla \rangle^\sigma z_3v_4\big|dxdt.
\end{align*}
Third term will be only considered, because remaining two terms can be handled similarly.

By using H\"older inequality, Lemma \ref{xsbbi-1} and Lemma \ref{xsbbi-2}, we get
\begin{align*}
&N_2^{2\al - d} \int\int_{\mathbb{R} \times \mathbb{R}^d} \big|z^k_{1}z^l_{2}\langle \nabla \rangle^\sigma z_3v_4\big|dxdt \lesssim N_2^{2\al - d} \|z^k_{1}\langle \nabla \rangle^\sigma z_3\|_{L^2_{t,x}}\|z^l_{2}v_4\|_{L^2_{t,x}}\\
&\lesssim N_1^{\frac{d-1}{2}}N_2^{2\al - d + \frac{d-1}{2} + 2\ep}N_3^{1-\al + \sigma + 2\ep}\|z^k_{1}\|_{X^{0,b}}\|z^l_{2}\|_{X^{0,b}}\|z_3\|_{X^{0,b}}\|v_4\|_{X^{0,\frac 12 - \ep}}.
\end{align*}
Since $\|z^k_{1}\|_{X^{0,b}} = \|z_{1}\|_{X^{0,b}}$, $\|z^l_{2}\|_{X^{0,b}} = \|z_{2}\|_{X^{0,b}}$ and $\sum_{k,l}|C_{k,l}| \leq C$, we have
\begin{align*}
&\quad \sum_{k,l \in \mathbb{Z}^d}C_{k,l} N_2^{2\al - d} \int\int_{\mathbb{R} \times \mathbb{R}^d} \big|\langle \nabla \rangle^\sigma  ((z^k_{1}z^l_{2})z_3)v_4\big|dxdt\\
&\lesssim N_1^{\frac{d-1}{2}}N_2^{2\al - d + \frac{d-1}{2} + 2\ep}N_3^{1-\al + \sigma + 2\ep}\|z_{1}\|_{X^{0,b}}\|z_{2}\|_{X^{0,b}}\|z_3\|_{X^{0,b}}\|v_4\|_{X^{0,\frac 12 - \ep}}.
\end{align*}
And by using Bernstein inequality carrying out summation in $N_1$, we get
\begin{align*}
&\quad \sum_{N_1 \ll N_2}\sum_{k,l \in \mathbb{Z}^d}C_{k,l} N_2^{2\al - d} \int\int_{\mathbb{R} \times \mathbb{R}^d} \big|\langle \nabla \rangle^\sigma  ((z^k_{1}z^l_{2})z_3)v_4\big|dxdt\\
&\lesssim N_2^{2\al - 1 - 2s + 2\ep}N_3^{\al - 1 + \sigma - s + 2\ep}\|z\|_{X^{s,b}}\|z_{2}\|_{X^{s,b}}\|z_3\|_{X^{s,b}}\|v_4\|_{X^{0,\frac 12 - \ep}}.
\end{align*}
Since $N_3 \gg N_3^{\frac {\al - 1}{2\al -1}} \gg N_2 \gg N_1$ and $2\al - 1 + s + 2\ep>0$, from summation in $N_2$, we obtain
\begin{align*}
&\quad \sum_{N_3^{\frac {\al - 1}{2\al -1}} \gg N_2 \gg N_1} \eqref{term2} \lesssim N_3^{-2s\frac {\al - 1}{2\al - 1} + \sigma - s + \frac{6\al - 4}{2\al - 1}\ep}\|z\|_{X^{s,b}}\|z\|_{X^{s,b}}\|z_3\|_{X^{s,b}}\|v_4\|_{X^{0,\frac 12 - \ep}}.
\end{align*}
In order to make summation in $N_3$ be finite, the power $-2s\frac {\al - 1}{2\al - 1} + \sigma - s$ should be negative. So we need the condition
\[s > \sigma \frac {2\al - 1}{4\al -3}.\]
After carrying out summation in $N_3$ and applying Lemma \ref{xsb}, we have
\[\sum_{N_4 \sim N_3 \gg N_3^{\frac {\al - 1}{2\al -1}} \gg N_2 \gg N_1} \eqref{term2}  \lesssim T^{0-} \|\phi^\omega\|_{H^s}^3\|v_4\|_{X^{0,\frac 12 - \ep}}.\]

Therefore, from Lemma \ref{prbest-2}, we conclude that
\[\sum_{N_4 \sim N_3 \gg N_3^{\frac {\al - 1}{2\al -1}} \gg N_2 \gg N_1} \eqref{term2} \lesssim T^{0-}R^3\|v_4\|_{X^{0,\frac 12 - \ep}}\]
outside a set of probality
\[C\exp\big(-c\frac {R^2}{\|\phi\|_{H^s}^2}\big).\]

Subcase (2.ii.b) : $N_4 \sim N_3 \gg N_3^{\frac {\al - 1}{2\al -1}} \gg N_2 \sim N_1$.

This case is more delicate because $|\nabla|^{2\al - d}$ might be singular on Fourier support of $z_1z_2$. First we decompose $|\nabla|^{2\al - d}$ such that \[|\nabla|^{2\al - d} = \sum_N N^{2\al - d}\psi(|\nabla|/N),\]
with a cut-off $\psi$ supported in $A(1)$. Here $\psi(|\nabla|)$ is pseudo-differential operator defined by $\psi(|\nabla|)f = \mathcal{F}^{-1}(\psi(|\cdot|)\mathcal{F}f)$. Then we have
\begin{align*}
&\int\int_{\mathbb{R} \times \mathbb{R}^d} \langle \nabla \rangle^\sigma ((|x|^{-2\al}*z_1z_2)z_3)v_4 dxdt = \int\int_{\mathbb{R} \times \mathbb{R}^d} \langle \nabla \rangle^\sigma  (\sum_{-\infty}^{N_2} N^{2\al - d}\psi(|\nabla|/N)(z_1z_2)z_3)v_4dxdt.
\end{align*}
After that we decompose $z_1$ and $z_2$ into functions having Fourier supports in cubes of side length $2^{-2}N$. Let $\{Q\}$ be a collection of essentially disjoint cubes of side length $2^{-2}N$ covering $A(N_2)$. Let us define $z_{iQ}$ by $\widehat{z_{iQ}} = \chi_Q(\xi) \widehat{z_i}$ for $i=1,2$. Then we have $z_i = \sum_Q z_{iQ}$ for $i=1,2$. Since $N_1 \sim N_2$, we may restrict $Q \subset A(N_2)$. So, we have
\begin{align*}
&\quad \int\int_{\mathbb{R} \times \mathbb{R}^d} \big|\langle \nabla \rangle^\sigma  (N^{2\al - d}\psi(|\nabla|/N)(z_1z_2)z_3)v_4\big|dxdt\\
&\lesssim \sum_{Q,Q'}\int\int_{\mathbb{R} \times \mathbb{R}^d} \big|\langle \nabla \rangle^\sigma  (N^{2\al - d}\psi(|\nabla|/N)(z_{1Q}z_{2Q'})z_3)v_4\big|dxdt\\
&= \sum_{dist(Q,-Q') \leq 4N}\int\int_{\mathbb{R} \times \mathbb{R}^d} \big|\langle \nabla \rangle^\sigma  (N^{2\al - d}\psi(|\nabla|/N)(z_{1Q}z_{2Q'})z_3)v_4\big|dxdt.
\end{align*}
Here, the last equality follows from the fact that $\psi(|\nabla|/N)(z_{1Q}z_{2Q'}) = 0$ if $dist(Q,-Q') > 4N$. We observe that
\begin{align*}
&\quad \psi(|\nabla|/N)(z_{1Q}z_{2Q'}) = \int\int e^{ix\cdot(\xi + \eta)} \chi(\xi/N - \xi_0)\psi((\xi + \eta)/N)\chi(\eta/N - \eta_0)\widehat{z_{1Q}}\widehat{z_{2Q'}}d\xi d\eta
\end{align*}
for some $\xi_0, \eta_0 \in \mathbb{R}^d$ and $\chi$ supported in $B(0,1)$. Let us take the Fourier series expansion for $\Psi(\xi, \eta) = \chi(\xi - \xi_0)\psi(\xi + \eta)\chi(\eta - \eta_0)$ on the cube of side length $2\pi$ which contains the support of $\Psi$ to get
\[\chi(\xi - \xi_0)\psi(\xi + \eta)\chi(\eta - \eta_0) = \sum_{k,l \in \mathbb{Z}^d} C_{k,l} e^{i(k\cdot\xi + l\cdot\eta)}\]
with $\sum_{k,l}|C_{k,l}| \leq C$, independent of $\xi_0, \eta_0$. So, we have \[\psi(|\nabla|/N)(z_{1Q}z_{2Q'}) = \sum_{k,l \in \mathbb{Z}^d} C_{k,l}z^k_{1Q}z^l_{2Q'}\]
\text{where}
$
z_{1Q}^k = \int e^{2\pi ix\cdot \xi}e^{2\pi i k\cdot\xi}\widehat{z_{1Q}}(\xi)d\xi
\text{ and }z_{2Q'}^l = \int e^{2\pi i x\cdot \eta}e^{2\pi i l\cdot\eta}\widehat{z_{2Q'}}(\eta)d\eta.
$
Hence we obtain
\begin{align*}
&\quad \sum_{dist(Q,-Q') \leq 4N}\int\int_{\mathbb{R} \times \mathbb{R}^d} \big|\langle \nabla \rangle^\sigma  (N^{2\al - d}\psi(|\nabla|/N)(z_{1Q}z_{2Q'})z_3)v_4\big|dxdt\\
&=  \sum_{dist(Q,-Q') \leq 4N}\int\int_{\mathbb{R} \times \mathbb{R}^d} \big|\langle \nabla \rangle^\sigma  (N^{2\al - d}(\sum_{k,l \in \mathbb{Z}^d}C_{k,l}z^k_{1Q}z^l_{2Q'})z_3)v_4\big|dxdt\\
&\lesssim
\sum_{dist(Q,-Q') \leq 4N} \sum_{k,l \in \mathbb{Z}^d}C_{k,l} N^{2\al - d} \int\int_{\mathbb{R} \times \mathbb{R}^d} \big|\langle \nabla \rangle^\sigma  ((z^k_{1Q}z^l_{2Q'})z_3)v_4\big|dxdt.
\end{align*}

So we need to handle
\begin{align*}
N^{2\al - d} \int\int_{\mathbb{R} \times \mathbb{R}^d} &\big|  \langle \nabla \rangle^\sigma z^k_{1Q}z^l_{2Q'}z_3v_4\big|dxdt,\ N^{2\al - d} \int\int_{\mathbb{R} \times \mathbb{R}^d} \big|z^k_{1Q}\langle \nabla \rangle^\sigma z^l_{2Q'}z_3v_4\big|dxdt\\
&\text{and }N^{2\al - d} \int\int_{\mathbb{R} \times \mathbb{R}^d} \big|z^k_{1Q}z^l_{2Q'}\langle \nabla \rangle^\sigma z_3v_4\big|dxdt.
\end{align*}
Third term will be only considered, because remaining two terms can be handled similarly. By using H\"older inequality, Lemma \ref{xsbbi-1} and Lemma \ref{xsbbi-2}, we get
\begin{align*}
N^{2\al - d} &\int\int_{\mathbb{R} \times \mathbb{R}^d} \big|z^k_{1Q}z^l_{2Q'}\langle \nabla \rangle^\sigma z_3v_4\big|dxdt \lesssim N^{2\al - d} \|z^k_{1Q}\langle \nabla \rangle^\sigma z_3\|_{L^2_{t,x}}\|z^l_{2Q'}v_4\|_{L^2_{t,x}}\\
&\lesssim N^{2\al - 1 - \ep_1}N_2^{\ep_2}N_3^{1 - \al + \sigma + 2\ep}\|z^k_{1Q}\|_{X^{0,b}}\|z^l_{2Q'}\|_{X^{0,b}}\|z_3\|_{X^{0,b}}\|v_4\|_{X^{0,\frac 12 - \ep}}.
\end{align*}
Since $\|z^k_{1Q}\|_{X^{0,b}} = \|z_{1Q}\|_{X^{0,b}}$, $\|z^l_{2Q'}\|_{X^{0,b}} = \|z_{2Q'}\|_{X^{0,b}}$ and $\sum_{k,l}|C_{k,l}| \leq C$, we have
\begin{align*}
&\quad \sum_{k,l \in \mathbb{Z}^d}C_{k,l} N^{2\al - d} \int\int_{\mathbb{R} \times \mathbb{R}^d} \big|z^k_{1Q}z^l_{2Q'}\langle \nabla \rangle^\sigma z_3v_4\big|dxdt\\
&\lesssim N^{2\al - 1 - \ep_1}N_2^{\ep_2}N_3^{1 - \al + \sigma + 2\ep}\|z_{1Q}\|_{X^{0,b}}\|z_{2Q'}\|_{X^{0,b}}\|z_3\|_{X^{0,b}}\|v_4\|_{X^{0,\frac 12 - \ep}}.
\end{align*}
Thereafter we use Cauchy-Schwartz inequality, orthogonality and Bernstein inequality to get
\begin{align*}
&\sum_{dist(Q,-Q') \leq 4N} \sum_{k,l \in \mathbb{Z}^d}C_{k,l} N^{2\al - d} \int\int_{\mathbb{R} \times \mathbb{R}^d} \big|z^k_{1Q}z^l_{2Q'}\langle \nabla \rangle^\sigma z_3v_4\big|dxdt\\
&\lesssim N^{2\al - 1 - \ep_1}N_2^{\ep_2}N_3^{1 - \al + \sigma + 2\ep}\|z_{1}\|_{X^{0,b}}\|z_{2}\|_{X^{0,b}}\|z_3\|_{X^{0,b}}\|v_4\|_{X^{0,\frac 12 - \ep}}\\
&\lesssim N^{2\al - 1 - \ep_1}N_2^{-2s + \ep_2}N_3^{1 - \al + \sigma - s + 2\ep}\|z_{1}\|_{X^{s,b}}\|z_{2}\|_{X^{s,b}}\|z_3\|_{X^{s,b}}\|v_4\|_{X^{0,\frac 12 - \ep}}.
\end{align*}

Then summation in $N \lesssim N_2$ gives
\begin{align*}
&\quad \sum_{N=-\infty}^{N_2} \sum_{dist(Q,-Q') \leq 4N} \sum_{k,l \in \mathbb{Z}^d}C_{k,l} N^{2\al - d} \int\int_{\mathbb{R} \times \mathbb{R}^d} \big|z^k_{1Q}z^l_{2Q'}\langle \nabla \rangle^\sigma z_3v_4\big|dxdt\\
&\lesssim N_2^{2\al - 1 - 2s + 2\ep}N_3^{1 - \al + \sigma - s + 2\ep}\|z_{1}\|_{X^{s,b}}\|z_{2}\|_{X^{s,b}}\|z_3\|_{X^{s,b}}\|v_4\|_{X^{0,\frac 12 - \ep}}.
\end{align*}
Since $N_3 \gg N_3^{\frac {\al - 1}{2\al -1}} \gg N_2 \sim N_1$ and $2\al - 1 - 2s + 2\ep > 0$, we have
\begin{align*}
\sum_{N_3^{\frac {\al - 1}{2\al -1}} \gg N_2 \sim N_1} \sum_{N=-\infty}^{N_2} &\sum_{dist(Q,-Q') \leq 4N} \sum_{k,l \in \mathbb{Z}^d}C_{k,l} N^{2\al - d} \int\int_{\mathbb{R} \times \mathbb{R}^d} \big|\langle \nabla \rangle^\sigma  ((z^k_{1Q}z^l_{2Q'})z_3)v_4\big|dxdt\\
&\lesssim N_3^{-2s\frac {\al - 1}{2\al - 1} + \sigma - s + \frac{6\al - 4}{2\al -1}\ep}\|z\|_{X^{s,b}}\|z\|_{X^{s,b}}\|z_3\|_{X^{s,b}}\|v_4\|_{X^{0,\frac 12 - \ep}}.
\end{align*}
Thus, from Lemma \ref{xsb}, we obtain
\[\lesssim T^{0-}N_3^{-2s\frac {\al - 1}{2\al - 1} + \sigma - s + \frac{6\al - 4}{2\al -1}\ep}\|\phi^\omega\|_{H^s}^2 \|P_{N_3}\phi^\omega\|_{H^s}\|v_4\|_{X^{0,\frac 12 - \ep}}.\]
We can carry out summation in $N_3$, because the power($-2s\frac {\al - 1}{2\al - 1} + \sigma - s$) is negative.

Therefore, by using Lemma \ref{prbest-2}, we conclude that
\[\sum_{N_4 \sim N_3 \gg N_3^{\frac {\al - 1}{2\al -1}} \gg N_2 \sim N_1} \eqref{term2} \lesssim T^{0-}R^3\|v_4\|_{X^{0,\frac 12 - \ep}}\]
outside a set of probality
\[C\exp\big(-c\frac {R^2}{\|\phi\|_{H^s}^2}\big).\]

Subcase (2.ii.c) $N_4 \sim N_3 \gg N_2 \gg N_3^{\frac {\al - 1}{2\al -1}} \gg N_1$

Adopting method in Case (2.ii.a), it suffices to estimate
\begin{align*}
& \sum_{k,l \in \mathbb{Z}^d} C_{k,l} N_2^{2\al - d} \int\int_{\mathbb{R} \times \mathbb{R}^d} \Big|\langle \nabla \rangle^\sigma  (z^k_1z^l_2z_3)v_4\Big|dxdt.
\end{align*}
So we have to handle
\begin{align*}
 N_2^{2\al - d} \int\int_{\mathbb{R} \times \mathbb{R}^d} &\big|  \langle \nabla \rangle^\sigma z^k_{1}z^l_{2}z_3v_4\big|dxdt,\ N_2^{2\al - d} \int\int_{\mathbb{R} \times \mathbb{R}^d} \big|z^k_{1}\langle \nabla \rangle^\sigma z^l_{2}z_3v_4\big|dxdt\\
&\text{and }N_2^{2\al - d} \int\int_{\mathbb{R} \times \mathbb{R}^d} \big|z^k_{1}z^l_{2}\langle \nabla \rangle^\sigma z_3v_4\big|dxdt.
\end{align*}
Third term will be only considered, because remaining two terms can be handled similarly. By using H\"older inequality, Lemma \ref{xsbbi-1} and Lemma \ref{xsbbi-2}, we get
\begin{align*}
N_2^{2\al - d} &\int\int_{\mathbb{R} \times \mathbb{R}^d} \big|z^k_{1}z^l_{2}\langle \nabla \rangle^\sigma z_3v_4\big|dxdt \lesssim N_2^{2\al - d} \|z^l_2\|_{L^4_{t,x}}\|\langle \nabla \rangle^\sigma z_3\|_{L^4_{t,x}}\|z^k_1v_4\|_{L^2_{t,x}}\\
&\lesssim N_2^{2\al - d}N_1^{\frac {d-1}{2} + 2\ep}N_4^{\frac {1-\al}{2} + 2\ep} \|z_1^k\|_{X^{0,b}}\|z_2^l\|_{L^4_{t,x}}\|\langle \nabla \rangle^\sigma z_3\|_{L^4_{t,x}}\|v_4\|_{X^{0,\frac 12 - \ep}}
\end{align*}

Thereafter, from $\sum_{k,l}|C_{k,l}| \leq C$ (with $\|z^k_1\|_{X^{0,b}}=\|z_1\|_{X^{0,b}},\|z^l_2\|_{L^4_{t,x}}=\|z_2\|_{L^4_{t,x}}$), Bernstein inequality and Lemma \ref{xsb}, we obtain
\begin{align*}
&\quad \sum_{k,l \in \mathbb{Z}^d} C_{k,l} N_2^{2\al - d} \int\int_{\mathbb{R} \times \mathbb{R}^d} \big|z^k_{1}z^l_{2}\langle \nabla \rangle^\sigma z_3v_4\big|dxdt\\
&\lesssim T^{0-}N_2^{-s+2\al - d}N_3^{\sigma -s}N_1^{\frac {d-1}{2} + 2\ep -s}N_4^{\frac {1-\al}{2} + 2\ep}\|P_{N_1}\phi^\omega\|_{H^s}\|\langle \nabla \rangle^s z_2\|_{L^4_{t,x}}\|\langle \nabla \rangle^s z_3\|_{L^4_{t,x}}\|v_4\|_{X^{0,\frac 12 - \ep}}.
\end{align*}
Then summation in $N_1$ and $N_2$ yields
\begin{align*}
&\quad \sum_{N_2 = N_3^{\frac {\al - 1}{2\al -1}}}^{N_3}\sum^{N_3^{\frac {\al - 1}{2\al -1}}}_{N_1=1}\sum_{k,l \in \mathbb{Z}^d} C_{k,l} N_2^{2\al - d} \int\int_{\mathbb{R} \times \mathbb{R}^d} \big|z^k_{1}z^l_{2}\langle \nabla \rangle^\sigma z_3v_4\big|dxdt\\
&\lesssim T^{0-}N_3^{-s\big(\frac {4\al - 3}{2\al - 1}\big) + \frac {\al - 1}{2(2\al -1)}(2\al - d) + \sigma + \frac{6\al - 4}{2\al - 1}\ep}\|\phi^\omega\|_{H^s}\|\langle \nabla \rangle^s z_2\|_{L^4_{t,x}}\|\langle \nabla \rangle^s z_3\|_{L^4_{t,x}}\|v_4\|_{X^{0,\frac 12 - \ep}}.
\end{align*}

Hence, from Lemma \ref{prbest-2} and Lemma \ref{prbStr-1}, we conclude that
\[\sum_{N_4 \sim N_3 \gg N_2 \gg N_3^{\frac {\al - 1}{2\al -1}} \gg N_1} \eqref{term2} \lesssim T^{0-}R^3\|v_4\|_{X^{0,\frac 12 - \ep}}\]
outside a set of probability
\[C\exp\big(-c\frac {R^2}{\|\phi\|^2_{H^s}}\big) + C\exp\big(-c\frac {R^2}{T^{\frac 12}\|\phi\|^2_{H^s}}\big).\]

Subcase (2.ii.d) : $N_4 \sim N_3 \gg N_2 \gg N_1 \gg N_3^{\frac {\al - 1}{2\al -1}}$

H\"older inequality and Lemma \ref{mass-xsb} yield $\eqref{term2}$ is bounded by
\begin{align*}
\lesssim \|\langle \nabla \rangle^\sigma ((|x|^{-2\al}*(z_1z_2))z_3)\|_{L^{\frac 1{1-\ep}}_tL^{2}_x}\|v_4\|_{L^{\frac {1}{\ep}}_tL^2_x} \lesssim \|\langle \nabla \rangle^\sigma ((|x|^{-2\al}*(z_1z_2))z_3)\|_{L^{\frac 1{1-\ep}}_tL^{2}_x}\|v_4\|_{X^{0,1-b-\theta}}
\end{align*}
for some small positive $\ep$. Thereafter, from Fractional integration theorem, H\"older inequality and Lemma \ref{lei}, we obtain
\begin{align*}
\eqref{term2} &\lesssim \|\langle \nabla \rangle^{\sigma} z_1\|_{L^{\frac {3}{1-\ep}}_tL^{\frac {2d}{d - \frac {4\al}{3}}}_x}\|z_2\|_{L^{\frac {3}{1-\ep}}_tL^{\frac {2d}{d - \frac {4\al}{3}}}_x}\|z_3\|_{L^{\frac {3}{1-\ep}}_tL^{\frac {2d}{d - \frac {4\al}{3}}}_x}\|v_4\|_{X^{0,1-b-\theta}}\\
&+ \|z_1\|_{L^{\frac {3}{1-\ep}}_tL^{\frac {2d}{d - \frac {4\al}{3}}}_x}\|\langle \nabla \rangle^{\sigma}z_2\|_{L^{\frac {3}{1-\ep}}_tL^{\frac {2d}{d - \frac {4\al}{3}}}_x}\|z_3\|_{L^{\frac {3}{1-\ep}}_tL^{\frac {2d}{d - \frac {4\al}{3}}}_x}\|v_4\|_{X^{0,1-b-\theta}}\\
&+\|z_1\|_{L^{\frac {3}{1-\ep}}_tL^{\frac {2d}{d - \frac {4\al}{3}}}_x}\|z_2\|_{L^{\frac {3}{1-\ep}}_tL^{\frac {2d}{d - \frac {4\al}{3}}}_x}\|\langle \nabla \rangle^\sigma z_3\|_{L^{\frac {3}{1-\ep}}_tL^{\frac {2d}{d - \frac {4\al}{3}}}_x}\|v_4\|_{X^{0,1-b-\theta}}.
\end{align*}
Then we use Bernstein's inequality and carry out summation in $N_1, N_2$ to get
\[\sum_{N_2 = N_1}^{N_3} \sum_{N_1 = N_3^{\frac {\al - 1}{2\al -1}}}^{N_2} \eqref{term2} \lesssim N_3^{\sigma - \frac {4\al - 3}{2\al - 1}s} \|\langle \nabla \rangle^{s}z\|_{L^{\frac {3}{1-\ep}}_tL^{\frac {2d}{d - \frac {4\al}{3}}}_x}^2 \|\langle \nabla \rangle^{s}z_3\|_{L^{\frac {3}{1-\ep}}_tL^{\frac {2d}{d - \frac {4\al}{3}}}_x}\|v_4\|_{X^{0,1-b-\theta}}.\]

Therefore, from Lemma \ref{prbStr-1}, we conclude that
\[\sum_{N_4 \sim N_3 \gg N_2 \gg N_1 \gg N_3^{\frac {\al - 1}{2\al -1}}} \eqref{term2} \lesssim R^3\|v_4\|_{X^{0,1-b-\theta}}\]
outside a set of probability
\[\leq C\exp\big(-c\frac {R^2}{T^{\frac {2-2\ep}{3}}\|\phi\|^2_{H^s}}\big).\]
\noindent
3rd Term : $vzv$ term
\begin{align}\label{term3}
\big|\int\int_{\mathbb{R} \times \mathbb{R}^d} \langle \nabla \rangle^\sigma (|x|^{-2\al}*(v_1z_2)v_3)v_4 dxdt\big|.
\end{align}

We consider two cases separately
\begin{enumerate}[\quad i.]
\item $\max(N_1, N_3) \gtrsim N_2$
\item $\max(N_1, N_3) \ll N_2 \sim N_4$.
\end{enumerate}

Case (3.i) : $\max(N_1, N_3) \gtrsim N_2$

We assume $N_1 \geq N_3$, because the other case can be similarly handled. H\"older inequality, Lemma \ref{lei} and Lemma \ref{mass-xsb} yield
\begin{align*}
\quad \eqref{term3} &\lesssim \|\langle \nabla \rangle^\sigma ((|x|^{-2\al}*(v_1z_2)v_3)\|_{L^{\frac 1{1-\ep}}_tL^{2}_x}\|v_4\|_{L^{\frac {1}{\ep}}_tL^2_x}\\
&\lesssim \||x|^{-2\al}*(v_1z_2)\|_{L^{\frac {1}{1-\ep}}_tL^{\infty}_x}\|\langle \nabla \rangle^\sigma v_3\|_{L^\infty_tL^2_x}\|v_4\|_{X^{0,1-b-\theta}}\\
&\quad+ \|\langle \nabla \rangle^\sigma(|x|^{-2\al}*(v_1z_2))\|_{L^{\frac {2}{1-\ep}}_tL^{\frac {2d}{2\al + \al\ep}}_x}\|v_3\|_{L^{\frac {2}{1-\ep}}_tL^{\frac {2d}{d - 2\al - \al\ep}}_x}\|v_4\|_{X^{0,1-b-\theta}}.
\end{align*}
Thereafter we use Lemma \ref{frac bound}, H\"older inequality, Sobolev embedding and Lemma \ref{lei} to get
\begin{align*}
&\lesssim \big(\|v_1\|^{\frac 12}_{L^{\frac {2}{1-\ep}}_tL^{\frac {2d}{d-2\al-\al\ep}}_x}\|v_1\|^{\frac 12}_{L^{\frac {2}{1-\ep}}_tL^{\frac {2d}{d-2\al + \al\ep}}_x}\|z_2\|^{\frac 12}_{L^{\frac {2}{1-\ep}}_tL^{\frac {2d}{d-2\al-\al\ep}}_x}\|z_2\|^{\frac 12}_{L^{\frac {2}{1-\ep}}_tL^{\frac {2d}{d-2\al + \al\ep}}_x}\\
&\quad +\|\langle \nabla \rangle^\sigma(v_1z_2)\|_{L^{\frac {2}{1-\ep}}_tL^{\frac {2d}{2d - 2\al + \al\ep}}_x}\big)\times\|v_3\|_{X^{\sigma,b}}\|v_4\|_{X^{0,1-b-\theta}}.
\end{align*}
Then Lemma \ref{lei} and Lemma \ref{str-xsb} give
\begin{align*}
&\lesssim \|v_1\|_{X^{\sigma,b}}\|z_2\|^{\frac 12}_{L^{\frac {2}{1-\ep}}_tL^{\frac {2d}{d-2\al-\al\ep}}_x}\|z_2\|^{\frac 12}_{L^{\frac {2}{1-\ep}}_tL^{\frac {2d}{d-2\al + \al\ep}}_x}\|v_3\|_{X^{\sigma,b}}\|v_4\|_{X^{0,1-b-\theta}}\\
&\quad + \|\langle \nabla \rangle^\sigma v_1\|_{L^\infty_tL^2_x} \|z_2\|_{L^{\frac {2}{1-\ep}}_tL^{\frac {2d}{d - 2\al + \al\ep}}_x}\|v_3\|_{X^{\sigma,b}}\|v_4\|_{X^{0,1-b-\theta}}\\
&\quad + \|v_1\|_{L^\infty_tL^2_x} \|\langle \nabla \rangle^\sigma z_2\|_{L^{\frac {2}{1-\ep}}_tL^{\frac {2d}{d - 2\al + \al\ep}}_x}\|v_3\|_{X^{\sigma,b}}\|v_4\|_{X^{0,1-b-\theta}}.
\end{align*}
Since $N_1 \gtrsim N_2$, we have
\begin{align*}
&\quad \|v_1\|_{L^\infty_tL^2_x} \|\langle \nabla \rangle^\sigma z_2\|_{L^{\frac {2}{1-\ep}}_tL^{\frac {2d}{d - 2\al + \al\ep}}_x}\|v_3\|_{X^{\sigma,b}}\|v_4\|_{X^{0,1-b-\theta}}\\
&\lesssim \|\langle \nabla \rangle^\sigma v_1\|_{L^\infty_tL^2_x} \|z_2\|_{L^{\frac {2}{1-\ep}}_tL^{\frac {2d}{d - 2\al + \al\ep}}_x}\|v_3\|_{X^{\sigma,b}}\|v_4\|_{X^{0,1-b-\theta}}.
\end{align*}

Therefore, from Lemma \ref{prbStr-1}, we conclude that
\begin{align*}
\sum_{\max(N_1, N_3) \gtrsim N_2}\eqref{term3} \lesssim R\|v\|_{X^{\sigma,b}}^2\|v_4\|_{X^{0,1-b-\theta}}
\end{align*}
outside a set of probability
\[\leq C\exp\big(-c\frac {R^2}{T^{1-\ep}\|\phi\|_{H^{0+}}}\big).\]

Case (3. ii) : $\max(N_1, N_3) \ll N_2 \sim N_4$.

We assume $N_3 \geq N_1$, because other case can be handled similarly.
As in Case (2.ii.a), we shall deal with
\begin{align*}
&\sum_{k,l \in \mathbb{Z}^d} C_{k,l} N_2^{2\al - d} \int\int_{\mathbb{R} \times \mathbb{R}^d} \Big|\langle \nabla \rangle^\sigma  (v^k_1z^l_2v_3)v_4\Big|dxdt.
\end{align*}
So we need to estimate
\begin{align*}
N_2^{2\al - d} \int\int_{\mathbb{R} \times \mathbb{R}^d} &\big|\langle \nabla \rangle^\sigma v^k_{1}z^l_{2}v_3v_4\big|dxdt,\ N_2^{2\al - d} \int\int_{\mathbb{R} \times \mathbb{R}^d} \big|v^k_{1}\langle \nabla \rangle^\sigma z^l_{2}v_3v_4\big|dxdt\\
&\text{and }N_2^{2\al - d} \int\int_{\mathbb{R} \times \mathbb{R}^d} \big|v^k_{1}z^l_{2}\langle \nabla \rangle^\sigma v_3v_4\big|dxdt.
\end{align*}
We will consider second term only, because remaining two terms can be handled similarly. By using H\"older inequality, Lemma \ref{xsbbi-1} and Lemma \ref{xsbbi-2}, we get
\begin{align*}
N_2^{2\al - d} &\int\int_{\mathbb{R} \times \mathbb{R}^d} \big|v^k_{1}\langle \nabla \rangle^\sigma z^l_{2} v_3v_4\big|dxdt \lesssim N_2^{2\al - d} \|\langle \nabla \rangle^\sigma z^l_2\|_{L^{4}_{t,x}}\|v_3\|_{L^{4}_{t,x}}\|v^k_1v_4\|_{L^2_{t,x}}\\
&\lesssim N_2^{2\al - d}N_1^{\frac {d-1}{2} + 2\ep}N_4^{\frac {1-\al}{2} + 2\ep} \|v^k_1\|_{X^{0,b}}\|\langle \nabla \rangle^\sigma z^l_2\|_{L^{4}_{t,x}}\|v_3\|_{L^{4}_{t,x}}\|v_4\|_{X^{0,\frac 12 - \ep}}
\end{align*}

Thereafter we use $\sum_{k,l}|C_{k,l}| \leq C$ (with $\|v^k_1\|_{X^{0,b}}=\|v_1\|_{X^{0,b}},\|z^l_2\|_{L^{4}_{t,x}}=\|z_2\|_{L^{4}_{t,x}}$), Bernstein inequality and Lemma \ref{str-xsb} to obtain
\begin{align*}
&\quad N_2^{2\al - d}\sum_{k,l \in \mathbb{Z}^d} C_{k,l} \int\int_{\mathbb{R} \times \mathbb{R}^d} \big|v^k_{1}\langle \nabla \rangle^\sigma z^l_{2} v_3v_4\big|dxdt \\
&\lesssim N_2^{\sigma - s + 2\al - d}N_1^{\frac {d-1}{2} + 2\ep - \sigma}N_3^{-\sigma + \frac {d-\al}{4}}N_4^{\frac {1-\al}{2} + 2\ep}\|v_1\|_{X^{\sigma, b}}\|\langle \nabla \rangle^s z_2\|_{L^{4}_{t,x}}\|v_3\|_{X^{\sigma, b}}\|v_4\|_{X^{0,\frac 12 - \ep}}.
\end{align*}
Then we carry out summation in $N_1$
\begin{align*}
&\quad \sum_{N_1=1}^{N_3} N_2^{2\al - d}\sum_{k,l \in \mathbb{Z}^d} C_{k,l} \int\int_{\mathbb{R} \times \mathbb{R}^d} \big|v^k_{1}\langle \nabla \rangle^\sigma z^l_{2} v_3v_4\big|dxdt\\
&\lesssim N_2^{\sigma - s + 2\al - d + \frac {1-\al}{2} + 2\ep}N_3^{\frac {d-1}{2} + 2\ep - \sigma -\sigma + \frac {d-\al}{4}}\|v\|_{X^{\sigma, b}}\|\langle \nabla \rangle^s z_2\|_{L^{4}_{t,x}}\|v_3\|_{X^{\sigma, b}}\|v_4\|_{X^{0,\frac 12 - \ep}}.
\end{align*}
We observe that $-d + 2\al + \frac{1}{2} < s$ gives $\sigma - s + 2\al - d + \frac {1-\al}{2} + 2\ep  < 0,$
when $\sigma - \frac \al 2, \ep$ is sufficiently small. So if $\frac {d-1}{2} - 2\sigma + \frac {d-\al}{4} + 2\ep < 0$, then summation can be carried out over $N_2, N_3$. Otherwise it should be checked that $\sigma - s + 2\al - d + \frac {1-\al}{2} + 2\ep + \frac {d-1}{2} - 2\sigma + \frac {d-\al}{4} + 2\ep < 0$. Actually the following hold
\begin{align*}
&\quad \sigma - s + 2\al - d + \frac {1-\al}{2} + 2\ep + \frac {d-1}{2} - 2\sigma + \frac {d-\al}{4} + 2\ep\\
&= -\sigma -s + \frac 54\al - \frac d4  + 4\ep
< -\big(\frac {6\al - 4}{4\al -3}\big)\sigma + \frac 54\al - \frac d4 + 4\ep (\text{ because } s > \frac {2\al - 1}{4\al -3}\sigma)\\
&< -\big(\frac {6\al - 4}{4\al -3}\big)\sigma + \frac 34\al + 4\ep (\text{ because } d > 2\al) <-\frac {1}{4(4\al - 3)}\al + 4\ep (\text{ because } \sigma > \frac \al2) <0
\end{align*}
when $\ep$ is sufficiently small. Hence we have that
\[\sum_{\max(N_1, N_3) \ll N_2 \sim N_4} \eqref{term3} \lesssim \|v\|_{X^{\sigma, b}}\|\langle \nabla \rangle^s z\|_{L^{4}_{t,x}}\|v\|_{X^{\sigma, b}}\|v_4\|_{X^{0,\frac 12 - \ep}}.\]

Therefore, from Lemma \ref{prbest-3}, we conclude that
\[\sum_{\max(N_1, N_3) \ll N_2 \sim N_4}\eqref{term3} \lesssim R\|v\|_{X^{\sigma, b}}^2\|v_4\|_{X^{0,\frac 12 - \ep}}\]
outside a set of probability \[\leq C\exp\big(-c\frac {R^2}{T^{\frac 12}\|\phi\|_{H^s}^2}\big).\]
\\
\noindent
4th Term : $vvz$ term
\begin{align}\label{term4}
\Big|\int\int_{\mathbb{R} \times \mathbb{R}^d} \langle \nabla \rangle^\sigma ((|x|^{-2\al}*v_1v_2)z_3)v_4 dxdt\Big|.
\end{align}

We consider 3 cases separately
\begin{enumerate}[\quad i.]
\item $N_2 \gtrsim N_3$
\item $N_1 \ll N_2 \ll N_3 \sim N_4$
\item $N_1 \sim N_2 \ll N_3 \sim N_4$.
\end{enumerate}

Case (4.i) : $N_2 \gtrsim N_3$

H\"older inequality and Lemma \ref{mass-xsb} yield
\begin{align*}
\eqref{term4} &\lesssim \|\langle \nabla \rangle^\sigma ((|x|^{-2\al}*(v_1v_2)z_3)\|_{L^{\frac 1{1-\ep}}_tL^{2}_x}\|v_4\|_{L^{\frac {1}{\ep}}_tL^2_x}\\
&\lesssim \||x|^{-2\al}*(v_1v_2)\|_{L^{\frac {2}{1-\ep}}_tL^{\frac {2d}{2\al + \al\ep}}_x}\|\langle \nabla \rangle^\sigma z_3\|_{L^{\frac {2}{1-\ep}}_tL^{\frac {2d}{d - 2\al - \al\ep}}_x}\|v_4\|_{X^{0,1-b-\theta}}\\
&\quad+ \|\langle \nabla \rangle^\sigma(|x|^{-2\al}*(v_1v_2))\|_{L^{\frac {2}{1-\ep}}_tL^{\frac {2d}{2\al + \al\ep}}_x}\|z_3\|_{L^{\frac {2}{1-\ep}}_tL^{\frac {2d}{d - 2\al - \al\ep}}_x}\|v_4\|_{X^{0,1-b-\theta}}
\end{align*}
Then we use Sobolev embedding and H\"older inequality to get
\begin{align*}
&\lesssim \|v_1v_2\|_{L^{\frac {2}{1-\ep}}_tL^{\frac {2d}{2d - 2\al + \al\ep}}_x}\|\langle \nabla \rangle^\sigma z_3\|_{L^{\frac {2}{1-\ep}}_tL^{\frac {2d}{d - 2\al - \al\ep}}_x}\|v_4\|_{X^{0,1-b-\theta}}\\
&\quad+\|\langle \nabla \rangle^\sigma(v_1v_2)\|_{L^{\frac {2}{1-\ep}}_tL^{\frac {2d}{2d - 2\al + \al\ep}}_x}\|z_3\|_{L^{\frac {2}{1-\ep}}_tL^{\frac {2d}{d - 2\al - \al\ep}}_x}\|v_4\|_{X^{0,1-b-\theta}}.
\end{align*}
And Lemma \ref{lei} gives
\begin{align*}
&\lesssim \|v_1\|_{L^{\frac {2}{1-\ep}}_tL^{\frac {2d}{d - 2\al + \al\ep}}_x}\|v_2\|_{L^\infty_tL^2_x}\|\langle \nabla \rangle^\sigma z_3\|_{L^{\frac {2}{1-\ep}}_tL^{\frac {2d}{d - 2\al - \al\ep}}_x}\|v_4\|_{X^{0,1-b-\theta}}\\
&\quad +\|v_1\|_{L^{\frac {2}{1-\ep}}_tL^{\frac {2d}{d - 2\al + \al\ep}}_x}\|\langle \nabla \rangle^\sigma v_2\|_{L^\infty_tL^2_x} \|z_3\|_{L^{\frac {2}{1-\ep}}_tL^{\frac {2d}{d - 2\al - \al\ep}}_x}\|v_4\|_{X^{0,1-b-\theta}}\\
&\quad + \|\langle \nabla \rangle^\sigma v_1\|_{L^\infty_tL^2_x} \|v_2\|_{L^{\frac {2}{1-\ep}}_tL^{\frac {2d}{d - 2\al + \al\ep}}_x}\|z_3\|_{L^{\frac {2}{1-\ep}}_tL^{\frac {2d}{d - 2\al - \al\ep}}_x}\|v_4\|_{X^{0,1-b-\theta}}
\end{align*}

Therefore, from $N_2 \gtrsim N_3$, Lemma \ref{str-xsb} and Lemma \ref{prbStr-1}, we conclude that
\[\sum_{N_2 \gtrsim N_3}\eqref{term4} \lesssim R\|v\|_{X^{\sigma,b}}^2\|v_4\|_{X^{0,1-b-\theta}}\]
outside a set of probability
\[\leq C\exp\big(-c\frac {R^2}{T^{1-\ep}\|\phi\|_{H^{0+}}}\big).\]

Case (4. ii) : $N_1 \ll N_2 \ll N_3 \sim N_4$

Adopting method in Case (2.ii.a), we need to deal with
\begin{align*}
& \sum_{k,l \in \mathbb{Z}^d} C_{k,l} N_2^{2\al - d} \int\int_{\mathbb{R} \times \mathbb{R}^d} \Big|\langle \nabla \rangle^\sigma  (v^k_1v^l_2z_3)v_4\Big|dxdt.
\end{align*}
So we have to estimate
\begin{align*}
N_2^{2\al - d} \int\int_{\mathbb{R} \times \mathbb{R}^d} &\big|  \langle \nabla \rangle^\sigma v^k_{1}v^l_{2}z_3v_4\big|dxdt,\ N_2^{2\al - d} \int\int_{\mathbb{R} \times \mathbb{R}^d} \big|v^k_{1}\langle \nabla \rangle^\sigma v^l_{2}z_3v_4\big|dxdt\\
&\text{and }N_2^{2\al - d} \int\int_{\mathbb{R} \times \mathbb{R}^d} \big|v^k_{1}v^l_{2}\langle \nabla \rangle^\sigma z_3v_4\big|dxdt.
\end{align*}
Third term will be only considered, because remaining two terms can be handled similarly. By using H\"older inequality and Lemma \ref{xsbbi-2}, we get
\begin{align*}
N_2^{2\al - d} &\int\int_{\mathbb{R} \times \mathbb{R}^d} \big|v^k_{1}v^l_{2}\langle \nabla \rangle^\sigma z_3v_4\big|dxdt \lesssim N_2^{2\al - d} \|v^l_2\|_{L^{4}_{t,x}}\|\langle \nabla \rangle^\sigma z_3\|_{L^{4}_{t,x}}\|v^k_1v_4\|_{L^2_{t,x}}\\
&\lesssim N_2^{2\al - d}N_1^{\frac {d-1}{2} + 2\ep}N_4^{\frac {1-\al}{2} + 2\ep} \|v^k_1\|_{X^{0,b}}\| v^l_2\|_{L^{4}_{t,x}}\|\langle \nabla \rangle^\sigma z_3\|_{L^{4}_{t,x}}\|v_4\|_{X^{0,\frac 12 - \ep}}.
\end{align*}
Then $\sum_{k,l}|C_{k,l}| \leq C$ (with $\|v^k_1\|_{X^{0,b}}=\|v_1\|_{X^{0,b}},\|v^l_2\|_{L^{4}_{t,x}}=\|v_2\|_{L^{4}_{t,x}}$), Bernstein inequality and Lemma \ref{str-xsb} yield
\begin{align*}
&\quad \sum_{k,l \in \mathbb{Z}^d}C_{k,l} N_2^{2\al - d} \int\int_{\mathbb{R} \times \mathbb{R}^d} \big|v^k_{1}v^l_{2}\langle \nabla \rangle^\sigma z_3v_4\big|dxdt\\
&\lesssim N_2^{-\sigma + \frac {d-\al}{4} + 2\al - d}N_1^{\frac {d-1}{2} + 2\ep - \sigma}N_3^{\sigma - s}N_4^{\frac {1-\al}{2} + 2\ep}\|v_1\|_{X^{\sigma, b}}\|v_2\|_{X^{\sigma, b}}\|\langle \nabla \rangle^s z_3\|_{L^{4}_{t,x}}\|v_4\|_{X^{0,\frac 12 - \ep}}.
\end{align*}

Thereafter we carry out summation in $N_1$ to get
\begin{align*}
&\quad \sum_{N_1=1}^{N_2}\sum_{k,l \in \mathbb{Z}^d}C_{k,l} N_2^{2\al - d} \int\int_{\mathbb{R} \times \mathbb{R}^d} \big|v^k_{1}v^l_{2}\langle \nabla \rangle^\sigma z_3v_4\big|dxdt\\
&\lesssim N_2^{-2\sigma -\frac{d}{4} + \frac{7}{4}\alpha - \frac{1}{2} + 2\ep}N_3^{\sigma -s + \frac{1-\al}{2} + 2\ep}\|v\|_{X^{\sigma, b}}\|v_2\|_{X^{\sigma, b}}\|\langle \nabla \rangle^s z_3\|_{L^{4}_{t,x}}\|v_4\|_{X^{0,\frac 12 - \ep}}.
\end{align*}
Here we observe that $-2\sigma - \frac{d}{4} + \frac{7}{4}\al - \frac12 + 2\ep < \frac\al4-\frac12 + 2\ep<0$
when $\ep$ is sufficiently small. Hence one can carry out summation over $N_2$. For summation over $N_3$, it is necessary that $\sigma -s + \frac{1-\al}{2} + 2\ep < 0$, which is true when $s > \frac 12$ and $\ep$ is sufficently small.

Therefore, from Lemma \ref{prbStr-1}, we conclude that
\begin{align*}
\sum_{N_1 \ll N_2 \ll N_3 \sim N_4} \eqref{term4} &\lesssim \|v\|_{X^{\sigma, b}}^2\|\langle \nabla \rangle^s z\|_{L^{4}_{t,x}}\|v_4\|_{X^{0,\frac 12 - \ep}}\lesssim R\|v\|_{X^{\sigma, b}}^2\|v_4\|_{X^{0,\frac 12 - \ep}}
\end{align*}
outside a set of probability \[\leq C\exp\big(-c\frac {R^2}{T^{\frac 12}\|\phi\|_{H^s}^2}\big).\]

Case (4.iii) : $N_1 \sim N_2 \ll N_3 \sim N_4$

As in Case (2.ii.b), we consider
\begin{align*}
\sum_{dist(Q,-Q') \leq 4N} \sum_{k,l \in \mathbb{Z}^d}C_{k,l} N^{2\al - d} \int\int_{\mathbb{R} \times \mathbb{R}^d} \big|\langle \nabla \rangle^\sigma  ((v^k_{1Q}v^l_{2Q'})z_3)v_4\big|dxdt.
\end{align*}
So we have to estimate
\begin{align*}
N^{2\al - d} \int\int_{\mathbb{R} \times \mathbb{R}^d} & \big|  \langle \nabla \rangle^\sigma v^k_{1Q}v^l_{2Q'}z_3v_4\big|dxdt,\ N^{2\al - d} \int\int_{\mathbb{R} \times \mathbb{R}^d} \big|v^k_{1Q}\langle \nabla \rangle^\sigma v^l_{2Q'}z_3v_4\big|dxdt\\
&\text{and }N^{2\al - d} \int\int_{\mathbb{R} \times \mathbb{R}^d} \big|v^k_{1Q}v^l_{2Q'}\langle \nabla \rangle^\sigma z_3v_4\big|dxdt.
\end{align*}
Third term will be only considered, because remaining two terms can be handled similarly. By using H\"older inequality, Lemma \ref{xsbbi-1} and Lemma \ref{xsbbi-2}, we get
\begin{align*}
&N^{2\al - d} \int\int_{\mathbb{R} \times \mathbb{R}^d} \big|v^k_{1Q}v^l_{2Q'}\langle \nabla \rangle^\sigma z_3v_4\big|dxdt \lesssim N^{2\al - d} \|v^l_{2Q'} \langle \nabla \rangle^\sigma z_3\|_{L^{2}_{t,x}}\|v^k_{1Q}v_4\|_{L^2_{t,x}}\\
&\lesssim N^{2\al - d}N^{d-1 + 2\ep}N_3^{1-\al + 2\ep} \|v^k_{1Q}\|_{X^{0,b}}\| v^l_{2Q'}\|_{X^{0,b}}\|\langle \nabla \rangle^\sigma z_3\|_{X^{0,b}}\|v_4\|_{X^{0,\frac 12 - \ep}}.
\end{align*}
Since $\sum_{k,l}|C_{k,l}| \leq C$ (with $\|v^k_{1Q}\|_{X^{0,b}} = \|v_{1Q}\|_{X^{0,b}}$, $\|v^l_{2Q'}\|_{X^{0,b}} = \|v_{2Q'}\|_{X^{0,b}}$), we have
\begin{align*}
&\quad \sum_{k,l \in \mathbb{Z}^d}C_{k,l} N^{2\al - d} \int\int_{\mathbb{R} \times \mathbb{R}^d} \big|v^k_{1Q}v^l_{2Q'}\langle \nabla \rangle^\sigma z_3v_4\big|dxdt\\
&\lesssim N^{2\al - d}N^{d-1 + 2\ep}N_3^{1-\al + 2\ep} \|v_{1Q}\|_{X^{0,b}}\| v_{2Q'}\|_{X^{0,b}}\|\langle \nabla \rangle^\sigma z_3\|_{X^{0,b}}\|v_4\|_{X^{0,\frac 12 - \ep}}.
\end{align*}
Thereafter we use Cauchy-Schwartz inequality, orthogonality and Bernstein inequality to obtain
\begin{align*}
&\quad \sum_{dist(Q,-Q') \leq 4N} \sum_{k,l \in \mathbb{Z}^d}C_{k,l} N^{2\al - d} \int\int_{\mathbb{R} \times \mathbb{R}^d} \big|\langle \nabla \rangle^\sigma  ((v^k_{1Q}v^l_{2Q'})z_3)v_4\big|dxdt\\
&\lesssim N^{2\al - d}N^{d-1 + 2\ep}N_3^{1-\al + 2\ep} \|v_{1}\|_{X^{0,b}}\| v_{2}\|_{X^{0,b}}\|\langle \nabla \rangle^\sigma z_3\|_{X^{0,b}}\|v_4\|_{X^{0,\frac 12 - \ep}}\\
&\lesssim N^{2\al-1 + 2\ep}N_2^{-2\sigma}N_3^{1-\al + \sigma - s + 2\ep} \|v_1\|_{X^{\sigma,b}}\|v_2\|_{X^{\sigma,b}}\|\langle \nabla \rangle^s z_3\|_{X^{0,b}}\|v_4\|_{X^{0,\frac 12 - \ep}}.
\end{align*}

Since $N \leq N_1 \sim N_2 \ll N_4$, we have
\begin{align*}
\sum_{N_2=1}^{N_3}\sum_{N=-\infty}^{N_2}&\sum_{dist(Q,-Q') \leq 4N} \sum_{k,l \in \mathbb{Z}^d}C_{k,l} N^{2\al - d} \int\int_{\mathbb{R} \times \mathbb{R}^d} \big|\langle \nabla \rangle^\sigma  ((v^k_{1Q}v^l_{2Q'})z_3)v_4\big|dxdt\\
&\lesssim N_3^{\al - 1 -\sigma -s + 4\ep} \|v\|_{X^{\sigma,b}}\|v\|_{X^{\sigma,b}}\|\langle \nabla \rangle^s z_3\|_{X^{0,b}}\|v_4\|_{X^{0,\frac 12 - \ep}}
\end{align*}
Then, from Lemma \ref{xsb}, we get
\[\lesssim T^{0-}N_3^{\al - 1 -\sigma -s + 4\ep} \|P_{N_3}\phi^\omega\|_{H^s}\|v\|_{X^{\sigma,b}}\|v\|_{X^{\sigma,b}}\|v_4\|_{X^{0,\frac 12 - \ep}}.\]
Because $\al - 1 -\sigma -s + 2\ep_1 <0$, we have finite summation of $N_3$.

Hence, by using Lemma \ref{prbest-1}, we conclude that
\[\sum_{N_1 \sim N_2 \ll N_3 \sim N_4}\eqref{term4} \lesssim T^{0-}R\|v\|_{X^{\sigma,b}}^2\|v_4\|_{X^{0,\frac 12 - \ep}}\] outside a set of probality
\[C\exp\big(-c\frac {R^2}{\|\phi\|_{H^s}^2}\big).\]
\noindent
5th Term : $vzz$ term
\begin{align}\label{term5}
\Big|\int\int_{\mathbb{R} \times \mathbb{R}^d} \langle \nabla \rangle^\sigma (|x|^{-2\al}*(v_1z_2))z_3v_4 dxdt\Big|.
\end{align}

We condiser 4 cases separately
\begin{enumerate}[\quad i.]
\item $N_1 \gtrsim \max(N_2,N_3)$
\item $N_1 \ll N_2 \sim  N_3$
\item $N_1, N_2 \ll N_3$
\item $N_1, N_3 \ll N_2$.
\end{enumerate}

Case (5.i) : $N_1 \gtrsim \max(N_2,N_3)$.

H\"older inequality, Lemma \ref{lei} and Lemma \ref{mass-xsb} give
\begin{align*}
\eqref{term5} &\lesssim \|\langle \nabla \rangle^\sigma ((|x|^{-2\al}*(v_1z_2)z_3)\|_{L^{\frac 1{1-\ep}}_tL^{2}_x}\|v_4\|_{L^{\frac {1}{\ep}}_tL^2_x}\\
&\lesssim \||x|^{-2\al}*(v_1z_2)\|_{L^{\frac {2}{1-\ep}}_tL^{\frac {2d}{2\al + \al\ep}}_x}\|\langle \nabla \rangle^\sigma z_3\|_{L^{\frac {2}{1-\ep}}_tL^{\frac {2d}{d - 2\al - \al\ep}}_x}\|v_4\|_{X^{0,1-b-\theta}}\\
&\quad+ \|\langle \nabla \rangle^\sigma(|x|^{-2\al}*(v_1z_2))\|_{L^{\frac {2}{1-\ep}}_tL^{\frac {2d}{2\al + \al\ep}}_x}\|z_3\|_{L^{\frac {2}{1-\ep}}_tL^{\frac {2d}{d - 2\al - \al\ep}}_x}\|v_4\|_{X^{0,1-b-\theta}}.
\end{align*}
Thereafter we use Sobolev embedding to obtain
\begin{align*}
&\lesssim \|v_1z_2\|_{L^{\frac {2}{1-\ep}}_tL^{\frac {2d}{2d - 2\al + \al\ep}}_x}\|\langle \nabla \rangle^\sigma z_3\|_{L^{\frac {2}{1-\ep}}_tL^{\frac {2d}{d - 2\al - \al\ep}}_x}\|v_4\|_{X^{0,1-b-\theta}}\\
&\quad+\|\langle \nabla \rangle^\sigma(v_1z_2)\|_{L^{\frac {2}{1-\ep}}_tL^{\frac {2d}{2d - 2\al + \al\ep}}_x}\|z_3\|_{L^{\frac {2}{1-\ep}}_tL^{\frac {2d}{d - 2\al - \al\ep}}_x}\|v_4\|_{X^{0,1-b-\theta}}
\end{align*}
Then H\"older inequality and Lemma \ref{lei} yield
\begin{align*}
&\lesssim \|v_1\|_{L^\infty_tL^2_x}\|z_2\|_{L^{\frac {2}{1-\ep}}_tL^{\frac {2d}{d - 2\al + \al\ep}}_x}\|\langle \nabla \rangle^\sigma z_3\|_{L^{\frac {2}{1-\ep}}_tL^{\frac {2d}{d - 2\al - \al\ep}}_x}\|v_4\|_{X^{0,1-b-\theta}}\\
&\quad +\|v_1\|_{L^\infty_tL^2_x} \|\langle \nabla \rangle^\sigma z_2\|_{L^{\frac {2}{1-\ep}}_tL^{\frac {2d}{d - 2\al + \al\ep}}_x} \|z_3\|_{L^{\frac {2}{1-\ep}}_tL^{\frac {2d}{d - 2\al - \al\ep}}_x}\|v_4\|_{X^{0,1-b-\theta}}\\
&\quad + \|\langle \nabla \rangle^\sigma v_1\|_{L^\infty_tL^2_x} \|z_2\|_{L^{\frac {2}{1-\ep}}_tL^{\frac {2d}{d - 2\al + \al\ep}}_x}\|z_3\|_{L^{\frac {2}{1-\ep}}_tL^{\frac {2d}{d - 2\al - \al\ep}}_x}\|v_4\|_{X^{0,1-b-\theta}}.
\end{align*}

Therefore, from Lemma \ref{prbStr-1}, we conclude that
\begin{align*}
\sum_{N_1 \gtrsim \max(N_2,N_3)}\eqref{term5} &\lesssim R^2\|v\|_{X^{\sigma,b}}\|v_4\|_{X^{0,1-b-\theta}}
\end{align*}
outside a set of probability
\[\leq C\exp\big(-c\frac {R^2}{T^{1-\ep}\|\phi\|_{H^{0+}}}\big).\]

Case (5.ii) : $N_1 \ll N_2 \sim  N_3$

We assume $N_4 \geq N_1$, because other case can be handled similarly.
As in Case (2.ii.a), we shall deal with
\begin{align*}
&\sum_{k,l \in \mathbb{Z}^d} C_{k,l} N_2^{2\al - d} \int\int_{\mathbb{R} \times \mathbb{R}^d} \Big|\langle \nabla \rangle^\sigma  (v^k_1z^l_2z_3)v_4\Big|dxdt.
\end{align*}
So we need to handle
\begin{align*}
N_2^{2\al - d} \int\int_{\mathbb{R} \times \mathbb{R}^d} &\big|\langle \nabla \rangle^\sigma v^k_{1}z^l_{2}z_3v_4\big|dxdt,\ N_2^{2\al - d} \int\int_{\mathbb{R} \times \mathbb{R}^d} \big|v^k_{1}\langle \nabla \rangle^\sigma z^l_{2}z_3v_4\big|dxdt\\
&\text{and }N_2^{2\al - d} \int\int_{\mathbb{R} \times \mathbb{R}^d} \big|v^k_{1}z^l_{2}\langle \nabla \rangle^\sigma z_3v_4\big|dxdt.
\end{align*}
Second term will be only considered, because remaining two terms can be handled similarly. By using H\"older inequality and Lemma \ref{xsbbi-2}, we get
\begin{align*}
N_2^{2\al - d} &\int\int_{\mathbb{R} \times \mathbb{R}^d} \big|v^k_{1}\langle \nabla \rangle^\sigma z^l_{2} z_3v_4\big|dxdt \lesssim N_2^{2\al - d} \|\langle \nabla \rangle^\sigma z^l_2\|_{L^{4}_{t,x}}\|z_3\|_{L^{4}_{t,x}}\|v^k_1v_4\|_{L^2_{t,x}}\\
&\lesssim N_2^{2\al - d}N_1^{\frac {d-1}{2} + 2\ep}N_4^{\frac {1-\al}{2} + 2\ep} \|v^k_1\|_{X^{0,b}}\|\langle \nabla \rangle^\sigma z^l_2\|_{L^{4}_{t,x}}\|z_3\|_{L^{4}_{t,x}}\|v_4\|_{X^{0,\frac 12 - \ep}}
\end{align*}
Then from $\sum_{k,l}|C_{k,l}| \leq C$ (with $\|v^k_1\|_{X^{0,b}}=\|v_1\|_{X^{0,b}},\|z^l_2\|_{L^{4}_{t,x}}=\|z_2\|_{L^{4}_{t,x}}$) and Bernstein inequality, we obtain
\begin{align*}
&\quad \sum_{k,l \in \mathbb{Z}^d}C_{k,l} N_2^{2\al - d} \int\int_{\mathbb{R} \times \mathbb{R}^d} \big|v^k_{1}\langle \nabla \rangle^\sigma z^l_{2} z_3v_4\big|dxdt\\
&\lesssim N_2^{\sigma - s + 2\al - d}N_1^{\frac {d-1}{2} + 2\ep - \sigma}N_3^{-s}N_4^{\frac {1-\al}{2} + 2\ep}\|v_1\|_{X^{\sigma, b}}\|\langle \nabla \rangle^s z_2\|_{L^{4}_{t,x}}\|\langle \nabla \rangle^s z_3\|_{L^4_{t,x}}\|v_4\|_{X^{0,\frac 12 - \ep}}\\
&\lesssim N_2^{\sigma - s + 2\al - d}N_1^{\frac {d-1}{2} - \sigma + \frac {1-\al}{2} + 4\ep}N_3^{-s}\|v_1\|_{X^{\sigma, b}}\|\langle \nabla \rangle^s z_2\|_{L^{4}_{t,x}}\|\langle \nabla \rangle^s z_3\|_{L^4_{t,x}}\|v_4\|_{X^{0,\frac 12 - \ep}}.
\end{align*}
Last line follows from $N_1 \leq N_4$.

Thereafter we carray out summation in $N_1$
\begin{align*}
\sum_{N_1=1}^{N_2}\eqref{term5} \lesssim \left\{\begin{array}{c}
N_2^{\frac {3\al - d}{2} - 2s + 4\ep}\|v\|_{X^{\sigma, b}}\|\langle \nabla \rangle^s z_2\|_{L^{4}_{t,x}}\|\langle \nabla \rangle^s z_3\|_{L^4_{t,x}}\|v_4\|_{X^{0,\frac 12 - \ep}}\\
\qquad (\text{ when }\frac {d-1}{2} - \sigma + \frac {1-\al}{2} + 4\ep > 0)\\
N_2^{\sigma - 2s + 2\al - d}\|v\|_{X^{\sigma, b}}\|\langle \nabla \rangle^s z_2\|_{L^{4}_{t,x}}\|\langle \nabla \rangle^s z_3\|_{L^4_{t,x}}\|v_4\|_{X^{0,\frac 12 - \ep}}\\
\qquad (\text{ when }\frac {d-1}{2} - \sigma + \frac {1-\al}{2} + 4\ep \leq 0).
\end{array}\right.
\end{align*}
Since $\frac {3\al - d}{2} - 2s + 4\ep \leq \frac {\al}{2} - 2s + 4\ep <0$ for sufficiently small $\ep$ and $\sigma - 2s + 2\al - d < 0$, we have
\begin{align*}
\sum_{N_1 \ll N_2 \sim  N_3}\eqref{term5} \lesssim \|v\|_{X^{\sigma, b}}\|\langle \nabla \rangle^s z\|_{L^{4}_{t,x}}^2\|v_4\|_{X^{0,\frac 12 - \ep}}.
\end{align*}

Therefore, from Lemma \ref{prbStr-1}, we conclude that
\[\sum_{N_1 \ll N_2 \sim  N_3}\eqref{term5} \lesssim R^2\|v\|_{X^{\sigma, b}}\|v_4\|_{X^{0,\frac 12 - \ep}}\]
outside a set of probability \[\leq C\exp\big(-c\frac {R^2}{T^{\frac 12}\|\phi\|_{H^s}^2}\big).\]

Case (5.iii) : $N_1, N_2 \ll N_3$

We consider 5 cases separately
\begin{enumerate}[\qquad a.]
\item $N_4 \sim N_3 \gg N_3^{\frac {\al - 1}{2\al -1}} \gg N_2 \gg N_1$
\item $N_4 \sim N_3 \gg N_3^{\frac {\al - 1}{2\al -1}} \gg N_2 \sim N_1$
\item $N_4 \sim N_3 \gg N_2 \gg N_3^{\frac {\al - 1}{2\al -1}} \gg N_1$
\item $N_4 \sim N_3 \gg N_1 \gg N_3^{\frac {\al - 1}{2\al -1}} \gg N_2$
\item $N_4 \sim N_3 \gg N_2, N_1 \gg N_3^{\frac {\al - 1}{2\al -1}}$.
\end{enumerate}

Subcase (5.iii.a) : $N_4 \sim N_3 \gg N_3^{\frac {\al - 1}{2\al -1}} \gg N_2 \gg N_1$

Similarly to Case (2.ii.a), we need to estimate
\begin{align*}
&\sum_{k,l \in \mathbb{Z}^d} C_{k,l} N_2^{2\al - d} \int\int_{\mathbb{R} \times \mathbb{R}^d} \Big|\langle \nabla \rangle^\sigma  (v^k_1z^l_2z_3)v_4\Big|dxdt.
\end{align*}
Hence we shall deal with
\begin{align*}
N_2^{2\al - d} \int\int_{\mathbb{R} \times \mathbb{R}^d} &\big|  \langle \nabla \rangle^\sigma v^k_{1}z^l_{2}z_3v_4\big|dxdt,\ N_2^{2\al - d} \int\int_{\mathbb{R} \times \mathbb{R}^d} \big|v^k_{1}\langle \nabla \rangle^\sigma z^l_{2}z_3v_4\big|dxdt\\
&\text{and }N_2^{2\al - d} \int\int_{\mathbb{R} \times \mathbb{R}^d} \big|v^k_{1}z^l_{2}\langle \nabla \rangle^\sigma z_3v_4\big|dxdt.
\end{align*}
Third term will be only considered, because remaining two terms can be handled similarly. By using H\"older inequality, Lemma \ref{xsbbi-1} and Lemma \ref{xsbbi-2}, we have
\begin{align*}
&\quad N_2^{2\al - d} \int\int_{\mathbb{R} \times \mathbb{R}^d} \big|v^k_{1}z^l_{2}\langle \nabla \rangle^\sigma z_3v_4\big|dxdt \lesssim N_2^{2\al - d}\|v^k_1\langle \nabla \rangle^\sigma z_3\|_{L^2_{t,x}}\|z^l_2v_4\|_{L^2_{t,x}}\\
&\lesssim N_2^{2\al - d}N_1^{\frac{d-1}{2}}N_3^{\frac{1-\al}{2}}N_2^{\frac{d-1}{2}+2\ep}N_4^{\frac{1-\al}{2}+2\ep}\|v^k_1\|_{X^{0,b}}\|\langle \nabla \rangle^\sigma z_3\|_{X^{0,b}}\|z^l_2\|_{X^{0,b}}\|v_4\|_{X^{0,\frac{1}{2}-\ep}}.
\end{align*}
Then $\sum_{k,l}|C_{k,l}| \leq C$ (with $\|v^k_1\|_{X^{0,b}} = \|v_1\|_{X^{0,b}}, \|z^l_2\|_{X^{0,b}}=\|z_2\|_{X^{0,b}}$), Bernstein inequality and Lemma \ref{xsb} yield
\begin{align*}
&\quad \sum_{k,l \in \mathbb{Z}^d}C_{k,l} N_2^{2\al - d} \int\int_{\mathbb{R} \times \mathbb{R}^d} \big|v^k_{1}z^l_{2}\langle \nabla \rangle^\sigma z_3v_4\big|dxdt\\
&\lesssim N_1^{\frac {d-1}{2} - \sigma}N_3^{\frac {1-\al}{2} + \sigma - s} N_2^{\frac {d-1}{2} + 2\al - d - s + 2\ep}N_4^{\frac {1-\al}{2} + 2\ep}\|v_1\|_{X^{\sigma,b}}\|z_2\|_{X^{s,b}}\|z_3\|_{X^{s,b}}\|v_4\|_{X^{0,\frac{1}{2}-\ep}}\\
&\lesssim T^{0-} N_1^{\frac {d-1}{2} - \sigma}N_3^{\frac {1-\al}{2} + \sigma - s} N_2^{\frac {d-1}{2} + 2\al - d - s + 2\ep}N_4^{\frac {1-\al}{2} + 2\ep}\|v_1\|_{X^{\sigma,b}}\prod^3_{j=2}\|P_{N_j}\phi^\omega\|_{H^s}\|v_4\|_{X^{0,\frac12 - \ep}}.
\end{align*}

Thereafter we carry out summation in $N_1$, $N_2$
\begin{align*}
&\quad \sum_{N_2=1}^{N_3^{\frac{\al-1}{2\al-1}}}\sum_{N_1 = 1}^{N_2}\sum_{k,l \in \mathbb{Z}^d}C_{k,l} N_2^{2\al - d} \int\int_{\mathbb{R} \times \mathbb{R}^d} \big|v^k_{1}z^l_{2}\langle \nabla \rangle^\sigma z_3v_4\big|dxdt\\
&\lesssim T^{0-} N_3^{\sigma(\frac {\al}{2\al - 1}) - s(\frac{3\al-2}{2\al-1}) + 2\ep(\frac{3\al-2}{2\al-1})}\|v\|_{X^{\sigma,b}}\|\phi^\omega\|_{H^s}\|P_{N_3}\phi^\omega\|_{H^s}\|v_4\|_{X^{0,\frac12 - \ep}}.
\end{align*}
Since $s > \sigma\frac{2\al-1}{4\al -3}>\sigma\frac{\al}{3\al-2}$, summation in $N_3$ can be also carried out.

Therefore, from Lemma \ref{prbest-2}, we conclude that
\[\sum_{N_4 \sim N_3 \gg N_3^{\frac {\al - 1}{2\al -1}} \gg N_2 \gg N_1} \eqref{term5} \lesssim T^{0-}R^2\|v\|_{X^{\sigma,b}}\|v_4\|_{X^{0,\frac12 -2\ep}}\] outside a set of probability \[C\exp(-c\frac{R^2}{\|\phi\|^2_{H^s}}).\]

Subcase (5.iii.b) : $N_4 \sim N_3 \gg N_3^{\frac {\al - 1}{2\al -1}} \gg N_2 \sim N_1$

Similarly to Case (2.ii.b), we consider
\begin{align*}
\sum_{dist(Q,-Q') \leq 4N} \sum_{k,l \in \mathbb{Z}^d}C_{k,l} N^{2\al - d} \int\int_{\mathbb{R} \times \mathbb{R}^d} \big|\langle \nabla \rangle^\sigma  ((v^k_{1Q}z^l_{2Q'})z_3)v_4\big|dxdt.
\end{align*}
So we need to handle
\begin{align*}
N^{2\al - d} \int\int_{\mathbb{R} \times \mathbb{R}^d} &\big|  \langle \nabla \rangle^\sigma v^k_{1Q}z^l_{2Q'}z_3v_4\big|dxdt,\ N^{2\al - d} \int\int_{\mathbb{R} \times \mathbb{R}^d} \big|v^k_{1Q}\langle \nabla \rangle^\sigma z^l_{2Q'}z_3v_4\big|dxdt\\
&\text{and }N^{2\al - d} \int\int_{\mathbb{R} \times \mathbb{R}^d} \big|v^k_{1Q}z^l_{2Q'}\langle \nabla \rangle^\sigma z_3v_4\big|dxdt.
\end{align*}
Third term will be only considered, because remaining two terms can be handled similarly. H\"older inequality, Lemma \ref{xsbbi-1} and Lemma \ref{xsbbi-2} give
\begin{align*}
&N^{2\al - d} \int\int_{\mathbb{R} \times \mathbb{R}^d} \big|v^k_{1Q}z^l_{2Q'}\langle \nabla \rangle^\sigma z_3v_4\big|dxdt \lesssim N^{2\al - d} \|v^k_{1Q} \langle \nabla \rangle^\sigma z_3\|_{L^{2}_{t,x}}\|z^l_{2Q'}v_4\|_{L^2_{t,x}}\\
&\lesssim N^{2\al - d}N^{d-1 + 2\ep}N_3^{1-\al + 2\ep} \|v^k_{1Q}\|_{X^{0,b}}\| z^l_{2Q'}\|_{X^{0,b}}\|\langle \nabla \rangle^\sigma z_3\|_{X^{0,b}}\|v_4\|_{X^{0,\frac 12 - \ep}}.
\end{align*}
And from $\sum_{k,l}|C_{k,l}| \leq C$ (with $\|v^k_{1Q}\|_{X^{0,b}} = \|v_{1Q}\|_{X^{0,b}}$, $\|z^l_{2Q'}\|_{X^{0,b}} = \|z_{2Q'}\|_{X^{0,b}}$), we have
\begin{align*}
&\quad \sum_{k,l \in \mathbb{Z}^d}C_{k,l} \int\int_{\mathbb{R} \times \mathbb{R}^d} \big|v^k_{1Q}z^l_{2Q'}\langle \nabla \rangle^\sigma z_3v_4\big|dxdt\\
&\lesssim N^{2\al - d}N^{d - 1 + 2\ep}N_3^{1 - \al + 2\ep} \|v_{1Q}\|_{X^{0,b}}\| z_{2Q'}\|_{X^{0,b}}\|\langle \nabla \rangle^\sigma z_3\|_{X^{0,b}}\|v_4\|_{X^{0,\frac 12 - \ep}}.
\end{align*}
Then Cauchy-Schwartz inequality, orthogonality and Bernstein inequality yield
\begin{align*}
&\quad \sum_{N=-\infty}^{N_2} \sum_{dist(Q,-Q') \leq 4N} \sum_{k,l \in \mathbb{Z}^d}C_{k,l} N^{2\al - d} \int\int_{\mathbb{R} \times \mathbb{R}^d} \big|\langle \nabla \rangle^\sigma  ((v^k_{1Q}z^l_{2Q'})z_3)v_4\big|dxdt\\
&\lesssim \sum_{N=-\infty}^{N_2} N^{2\al - d}N^{d-1 + 2\ep}N_3^{1-\al + 2\ep} \|v_{1}\|_{X^{0,b}}\| z_{2}\|_{X^{0,b}}\|\langle \nabla \rangle^\sigma z_3\|_{X^{0,b}}\|v_4\|_{X^{0,\frac 12 - \ep}}\\
&\lesssim N_2^{2\al - 1 + 2\ep}N_3^{1-\al + 2\ep} \|v_{1}\|_{X^{0,b}}\| z_{2}\|_{X^{0,b}}\|\langle \nabla \rangle^\sigma z_3\|_{X^{0,b}}\|v_4\|_{X^{0,\frac 12 - \ep}}\\
&\lesssim N_2^{2\al -1 - \sigma - s + 2\ep}N_3^{1-\al + \sigma - s + 2\ep} \|v_{1}\|_{X^{\sigma,b}}\| \langle \nabla \rangle^s z_{2}\|_{X^{0,b}}\|\langle \nabla \rangle^s z_3\|_{X^{0,b}}\|v_4\|_{X^{0,\frac 12 - \ep}}
\end{align*}

Since $N_2 \ll N_3^{\frac {\al - 1}{2\al -1}}$ and $2\al -1 - \sigma - s + \ep_1 >0$, we can carry out summation in $N_2$ so that
\begin{align*}
\sum_{N_2=1}^{N_3^{\frac{\al - 1}{2\al - 1}}} &\sum_{N=-\infty}^{N_2} \sum_{dist(Q,-Q') \leq 4N} \sum_{k,l \in \mathbb{Z}^d}C_{k,l} N^{2\al - d} \int\int_{\mathbb{R} \times \mathbb{R}^d} \big|\langle \nabla \rangle^\sigma  ((v^k_{1Q}z^l_{2Q'})z_3)v_4\big|dxdt\\
&\lesssim N_3^{\sigma(\frac{\al}{2\al-1}) - s(\frac{3\al-2}{2\al -1}) + 2\ep(\frac{3\al-2}{2\al -1})}\|v\|_{X^{\sigma,b}}\| \langle \nabla \rangle^s z\|_{X^{0,b}}\|\langle \nabla \rangle^s z_3\|_{X^{0,b}}\|v_4\|_{X^{0,\frac 12 - \ep}}.
\end{align*}
Then Lemma \ref{xsb} yield
\begin{align*}
\lesssim T^{0-} N_3^{\sigma(\frac {\al}{2\al - 1}) - s(\frac{3\al-2}{2\al-1}) + 2\ep(\frac{3\al-2}{2\al -1})}\|v\|_{X^{\sigma,b}}\|\phi^\omega\|_{H^s}\|P_{N_3}\phi^\omega\|_{H^s}\|v_4\|_{X^{0,\frac12 - \ep}}.
\end{align*}
Since $s > \sigma\frac{2\al-1}{4\al -3}>\sigma\frac{\al}{3\al-2}$, summation in $N_3$ can be carried out.

Therefore, from Lemma \ref{prbest-2}, we conclude that
\[\sum_{N_4 \sim N_3 \gg N_3^{\frac {\al - 1}{2\al -1}} \gg N_2 \sim N_1}\eqref{term5} \lesssim T^{0-}R^2\|v\|_{X^{\sigma,b}}\|v_4\|_{X^{0,\frac12 -2\ep}}\]
outside a set of probability \[C\exp(-c\frac{R^2}{\|\phi\|^2_{H^s}}).\]

Subcase (5.iii.c) : $N_4 \sim N_3 \gg N_2 \gg N_3^{\frac {\al - 1}{2\al -1}} \gg N_1$

As in Case (2.ii.a), we need to estimate
\begin{align*}
&\sum_{k,l \in \mathbb{Z}^d} C_{k,l} N_2^{2\al - d} \int\int_{\mathbb{R} \times \mathbb{R}^d} \Big|\langle \nabla \rangle^\sigma  (v^k_1z^l_2z_3)v_4\Big|dxdt.
\end{align*}
So we have to deal
\begin{align*}
N_2^{2\al - d} \int\int_{\mathbb{R} \times \mathbb{R}^d} &\big|  \langle \nabla \rangle^\sigma v^k_{1}z^l_{2}z_3v_4\big|dxdt,\ N_2^{2\al - d} \int\int_{\mathbb{R} \times \mathbb{R}^d} \big|v^k_{1}\langle \nabla \rangle^\sigma z^l_{2}z_3v_4\big|dxdt\\
&\text{and }N_2^{2\al - d} \int\int_{\mathbb{R} \times \mathbb{R}^d} \big|v^k_{1}z^l_{2}\langle \nabla \rangle^\sigma z_3v_4\big|dxdt.
\end{align*}
Third term will be only considered, because remaining two terms can be handled similarly. By using H\"older inequality and Lemma \ref{xsbbi-2}, we have
\begin{align*}
&\quad N_2^{2\al - d} \int\int_{\mathbb{R} \times \mathbb{R}^d} \big|v^k_{1}z^l_{2}\langle \nabla \rangle^\sigma z_3v_4\big|dxdt \lesssim N_2^{2\al - d} \|z^l_2\|_{L^4_{t,x}}\|\langle \nabla \rangle^\sigma z_3\|_{L^4_{t,x}}\|v^k_1v_4\|_{L^2_{t,x}}\\
&\lesssim N_1^{\frac {d-1}{2} - \sigma + 2\ep}N_3^{\sigma - s}N_2^{2\al - d - s}N_4^{\frac {1-\al}{2} + 2\ep}\|v^k_1\|_{X^{\sigma,b}}\|z^l_2\|_{L^4_{t,x}}\|\langle \nabla \rangle^\sigma z_3\|_{L^4_{t,x}}\|v_4\|_{X^{0,\frac12 - \ep}}.
\end{align*}
Then, from $\sum_{k,l}|C_{k,l}| \leq C$ (with $\|v^k_1\|_{X^{\sigma,b}}=\|v_1\|_{X^{\sigma,b}}, \|z^l_2\|_{L^4_{t,x}}=\|z_2\|_{L^4_{t,x}}$) and Bernstein inequality, we obtain
\begin{align*}
&\quad \sum_{k,l \in \mathbb{Z}^d}C_{k,l} N_2^{2\al - d} \int\int_{\mathbb{R} \times \mathbb{R}^d} \big|v^k_{1}z^l_{2}\langle \nabla \rangle^\sigma z_3v_4\big|dxdt\\
&\lesssim T^{0-} N_1^{\frac {d-1}{2} - \sigma + 2\ep}N_3^{\sigma - s}N_2^{2\al - d - s}N_4^{\frac {1-\al}{2} + 2\ep}\|v_1\|_{X^{\sigma,b}}\prod^3_{j=2}\|\langle \nabla \rangle^s z_j\|_{L^4_{t,x}}\|v_4\|_{X^{0,\frac12 - \ep}}.
\end{align*}

Since $N_4 \sim N_3 \gg N_2 \gg N_3^{\frac {\al - 1}{2\al -1}} \gg N_1$, $\frac {d-1}{2} - \sigma + \ep_1>0$ and $2\al - d - s < 0$, we can carry out summation in $N_1$ and $N_2$ to get
\begin{align*}
&\quad \sum_{N_2 = N_3^{\frac {\al - 1}{2\al -1}}}^{N_3} \sum_{N_1 = 1}^{N_3^{\frac {\al - 1}{2\al -1}}}\sum_{k,l \in \mathbb{Z}^d}C_{k,l} N_2^{2\al - d} \int\int_{\mathbb{R} \times \mathbb{R}^d} \big|v^k_{1}z^l_{2}\langle \nabla \rangle^\sigma z_3v_4\big|dxdt\\
&\lesssim T^{0-} N_3^{\frac{(2\al -d)(\al-1)}{2(2\al-1)} + \sigma(\frac{\al}{2\al-1}) - s(\frac{3\al-2}{2\al-1}) + 2\ep(\frac{3\al-2}{2\al-1})}\|v\|_{X^{\sigma,b}}\|\langle \nabla \rangle^s z\|_{L^4_{t,x}}\|\langle \nabla \rangle^s z_3\|_{L^4_{t,x}}\|v_4\|_{X^{0,\frac12 - \ep}}.
\end{align*}
Since $\frac{(2\al -d)(\al-1)}{2(2\al-1)} + \sigma(\frac{\al}{2\al-1}) - s(\frac{3\al-2}{2\al-1}) + 2\ep(\frac{3\al-2}{2\al-1}) < 0$ for sufficiently small $\ep$, summation in $N_3$ can be also carried out.

Therefore, by using Lemma \ref{prbStr-1}, we conclude that
\[\sum_{N_4 \sim N_3 \gg N_2 \gg N_3^{\frac {\al - 1}{2\al -1}} \gg N_1} \eqref{term5} \lesssim T^{0-}R^2\|v\|_{X^{\sigma,b}}\|v_4\|_{X^{0,\frac12 -2\ep}}\] outside a set of probability \[C\exp(-c\frac{R^2}{T^{\frac12}\|\phi\|^2_{H^s}}).\]

Subcase(5.iii.d) : $N_4 \sim N_3 \gg N_1 \gg N_3^{\frac {\al - 1}{2\al -1}} \gg N_2$

Similarly to Case (2.ii.a), we need to handle
\begin{align*}
&\sum_{k,l \in \mathbb{Z}^d} C_{k,l} N_1^{2\al - d} \int\int_{\mathbb{R} \times \mathbb{R}^d} \Big|\langle \nabla \rangle^\sigma  (v^k_1z^l_2z_3)v_4\Big|dxdt.
\end{align*}
So we shall estimate
\begin{align*}
N_1^{2\al - d} \int\int_{\mathbb{R} \times \mathbb{R}^d} &\big|  \langle \nabla \rangle^\sigma v^k_{1}z^l_{2}z_3v_4\big|dxdt,\ N_1^{2\al - d} \int\int_{\mathbb{R} \times \mathbb{R}^d} \big|v^k_{1}\langle \nabla \rangle^\sigma z^l_{2}z_3v_4\big|dxdt\\
&\text{and }N_1^{2\al - d} \int\int_{\mathbb{R} \times \mathbb{R}^d} \big|v^k_{1}z^l_{2}\langle \nabla \rangle^\sigma z_3v_4\big|dxdt.
\end{align*}
Third term will be only considered, because remaining two terms can be handled similarly. By using H\"older inequality and Lemma \ref{xsbbi-2}, we have
\begin{align*}
&N_1^{2\al - d} \int\int_{\mathbb{R} \times \mathbb{R}^d} \big|v^k_{1}z^l_{2}\langle \nabla \rangle^\sigma z_3v_4\big|dxdt \lesssim N_1^{2\al - d} \|v^k_1\|_{L^q_tL^r_x}\|\langle \nabla \rangle^\sigma z_3\|_{L^{\frac{2q}{q-2}}_tL^{\frac{2r}{r-2}}_x}\|z^l_2v_4\|_{L^2_{t,x}}\\
&\lesssim N_1^{2\al - d}N_2^{\frac {d-1}{2} - s + 2\ep}N_4^{\frac {1-\al}{2} + 2\ep} \|v^k_1\|_{L^q_tL^r_x}\|\langle \nabla \rangle^\sigma z_3\|_{L^{\frac{2q}{q-2}}_tL^{\frac{2r}{r-2}}_x}\|z^l_2\|_{X^{0,b}}\|v_4\|_{X^{0,\frac12 -2\ep}},
\end{align*}
where $q = \frac{2(2-\al)}{3-2\al}$ when $d=3$ and $q=3$ when $d \geq 4$.

Then, from $\sum_{k,l}|C_{k,l}| \leq C$ (with $\|v^k_1\|_{L^q_tL^r_x}=\|v_1\|_{L^q_tL^r_x}, \|z^l_2\|_{X^{0,b}}=\|z_2\|_{X^{0,b}}$) and Lemma \ref{str-xsb}, we obtain
\begin{align*}
&\quad \sum_{k,l \in \mathbb{Z}^d}C_{k,l} N_1^{2\al - d} \int\int_{\mathbb{R} \times \mathbb{R}^d} \big|v^k_{1}z^l_{2}\langle \nabla \rangle^\sigma z_3v_4\big|dxdt\\
&\lesssim N_1^{2\al - d + \frac{2-\al}{q}}N_2^{\frac {d-1}{2} + 2\ep}N_3^{\sigma}N_4^{\frac {1-\al}{2} + 2\ep}\|v_1\|_{X^{0,b}}\|z_2\|_{X^{0,b}}\|\langle \nabla \rangle^\sigma z_3\|_{L^{\frac{2q}{q-2}}_tL^{\frac{2r}{r-2}}_x}\|v_4\|_{X^{0,\frac12 - \ep}}.
\end{align*}
Thereafter we use Bernstein inequalityand Lemma \ref{xsb} to get
\begin{align*}
&\lesssim N_1^{2\al - d + \frac{2-\al}{q}-\sigma}N_2^{\frac {d-1}{2} - s + 2\ep}N_3^{\sigma - s}N_4^{\frac {1-\al}{2} + 2\ep}\|v_1\|_{X^{\sigma,b}}\|z_2\|_{X^{s,b}}\|\langle \nabla \rangle^s z_3\|_{L^{\frac{2q}{q-2}}_tL^{\frac{2r}{r-2}}_x}\|v_4\|_{X^{0,\frac12 - \ep}}\\
&\lesssim T^{0-} N_1^{2\al - d + \frac{2-\al}{q}-\sigma}N_2^{\frac {d-1}{2} - s + 2\ep}N_3^{\sigma - s}N_4^{\frac {1-\al}{2} + 2\ep}\|v_1\|_{X^{\sigma,b}}\|P_{N_2}\phi^\omega\|_{H^s}\|\langle \nabla \rangle^s z_3\|_{L^{\frac{2q}{q-2}}_tL^{\frac{2r}{r-2}}_x}\|v_4\|_{X^{0,\frac12 - \ep}}.
\end{align*}
Since $N_4 \sim N_3 \gg N_1 \gg N_3^{\frac {\al - 1}{2\al -1}} \gg N_2$, we can carry out summation in $N_1$ and $N_2$ as follows
\begin{align*}
&\quad \sum_{N_1 = N_3^{\frac {\al - 1}{2\al -1}}}^{N_3}\sum_{N_2=1}^{N_3^{\frac {\al - 1}{2\al -1}}}N_2^{2\al - d} \int\int_{\mathbb{R} \times \mathbb{R}^d} \big|v^k_{1}z^l_{2}\langle \nabla \rangle^\sigma z_3v_4\big|dxdt\\
&\lesssim T^{0-}N_3^{\frac{\al-1}{2\al-1}(\frac{2-\al}{q}+\al-\frac{d}{2}) + \frac{\al}{2\al-1}\sigma - \frac{3\al-2}{2\al-1}s + 2\frac{3\al-2}{2\al-1}\ep}\|v\|_{X^{\sigma,b}}\|\phi^\omega\|_{H^s}\|\langle \nabla \rangle^s z_3\|_{L^{\frac{2q}{q-2}}_tL^{\frac{2r}{r-2}}_x}\|v_4\|_{X^{0,\frac12 -\ep}}
\end{align*}

After carrying out summation in $N_3$, we apply Lemma \ref{prbest-2} and Lemma \ref{prbStr-1} to get
\[\sum_{N_4 \sim N_3 \gg N_1 \gg N_3^{\frac {\al - 1}{2\al -1}} \gg N_2} \eqref{term5} \lesssim T^{0-}R^2\|v\|_{X^{\sigma,b}}\|v_4\|_{X^{0,\frac12 - \ep}}\] outside a set of probability \[C\exp(-c\frac{R^2}{\|\phi\|^2_{H^s}}) + C\exp(-c\frac{R^2}{T^{\frac{q-2}{q}}\|\phi\|^2_{H^s}}).\]

Subcase(5.iii.e) : $N_4 \sim N_3 \gg N_2, N_1 \gg N_3^{\frac {\al - 1}{2\al -1}}$

From H\"older inequality, Lemma \ref{lei} and Lemma \ref{mass-xsb}, we obtain
\begin{align*}
\eqref{term5} &\lesssim \|\langle \nabla \rangle^\sigma ((|x|^{-2\al}*(v_1z_2)z_3)\|_{L^{\frac 1{1-\ep}}_tL^{2}_x}\|v_4\|_{L^{\frac {1}{\ep}}_tL^2_x}\\
&\lesssim \||x|^{-2\al}*(v_1z_2)\|_{L^{\frac {2}{1-\ep}}_tL^{\frac {2d}{2\al + \al\ep}}_x}\|\langle \nabla \rangle^\sigma z_3\|_{L^{\frac {2}{1-\ep}}_tL^{\frac {2d}{d - 2\al - \al\ep}}_x}\|v_4\|_{X^{0,1-b-\theta}}\\
&\quad+ \|\langle \nabla \rangle^\sigma(|x|^{-2\al}*(v_1z_2))\|_{L^{\frac {2}{1-\ep}}_tL^{\frac {2d}{2\al + \al\ep}}_x}\|z_3\|_{L^{\frac {2}{1-\ep}}_tL^{\frac {2d}{d - 2\al - \al\ep}}_x}\|v_4\|_{X^{0,1-b-\theta}}.
\end{align*}
And by using Sobolev embedding, we have
\begin{align*}
&\lesssim \|v_1z_2\|_{L^{\frac {2}{1-\ep}}_tL^{\frac {2d}{2d - 2\al + \al\ep}}_x}\|\langle \nabla \rangle^\sigma z_3\|_{L^{\frac {2}{1-\ep}}_tL^{\frac {2d}{d - 2\al - \al\ep}}_x}\|v_4\|_{X^{0,1-b-\theta}}\\
&\quad+\|\langle \nabla \rangle^\sigma(v_1z_2)\|_{L^{\frac {2}{1-\ep}}_tL^{\frac {2d}{2d - 2\al + \al\ep}}_x}\|z_3\|_{L^{\frac {2}{1-\ep}}_tL^{\frac {2d}{d - 2\al - \al\ep}}_x}\|v_4\|_{X^{0,1-b-\theta}}.
\end{align*}
Then H\"older inequality and Lemma \ref{lei} yield
\begin{align*}
&\lesssim \|v_1\|_{L^\infty_tL^2_x}\|z_2\|_{L^{\frac {2}{1-\ep}}_tL^{\frac {2d}{d - 2\al + \al\ep}}_x}\|\langle \nabla \rangle^\sigma z_3\|_{L^{\frac {2}{1-\ep}}_tL^{\frac {2d}{d - 2\al - \al\ep}}_x}\|v_4\|_{X^{0,1-b-\theta}}\\
&\quad +\|v_1\|_{L^\infty_tL^2_x} \|\langle \nabla \rangle^\sigma z_2\|_{L^{\frac {2}{1-\ep}}_tL^{\frac {2d}{d - 2\al + \al\ep}}_x} \|z_3\|_{L^{\frac {2}{1-\ep}}_tL^{\frac {2d}{d - 2\al - \al\ep}}_x}\|v_4\|_{X^{0,1-b-\theta}}\\
&\quad + \|\langle \nabla \rangle^\sigma v_1\|_{L^\infty_tL^2_x} \|z_2\|_{L^{\frac {2}{1-\ep}}_tL^{\frac {2d}{d - 2\al + \al\ep}}_x}\|z_3\|_{L^{\frac {2}{1-\ep}}_tL^{\frac {2d}{d - 2\al - \al\ep}}_x}\|v_4\|_{X^{0,1-b-\theta}}
\end{align*}
Thereafter we use Bernstein inequality and Lemma \ref{str-xsb} to get
\begin{align*}
&\lesssim N_1^{-\sigma}N_2^{-s}N_3^{\sigma - s}\|v_1\|_{X^{\sigma,b}}\|\langle \nabla \rangle^s z_2\|_{L^{\frac {2}{1-\ep}}_tL^{\frac {2d}{d - 2\al + \al\ep}}_x}\|\langle \nabla \rangle^s z_3\|_{L^{\frac {2}{1-\ep}}_tL^{\frac {2d}{d - 2\al - \al\ep}}_x}\|v_4\|_{X^{0,1-b-\theta}}
\end{align*}
Now we carry out summation in $N_1$, $N_2$ and $N_3$
\begin{align*}
\sum_{N_4 \sim N_3 \gg N_2, N_1 \gg N_3^{\frac {\al - 1}{2\al -1}}} \eqref{term5} \lesssim \|v\|_{X^{\sigma,b}}\|\langle \nabla \rangle^s z\|_{L^{\frac {2}{1-\ep}}_tL^{\frac {2d}{d - 2\al + \al\ep}}_x}\|\langle \nabla \rangle^s z\|_{L^{\frac {2}{1-\ep}}_tL^{\frac {2d}{d  - 2\al - \al\ep}}_x}\|v_4\|_{X^{0,1-b-\theta}}.
\end{align*}

Therefore, from Lemma \ref{prbStr-1}, we conclude that
\begin{align*}
\sum_{N_4 \sim N_3 \gg N_2, N_1 \gg N_3^{\frac {\al - 1}{2\al -1}}} \eqref{term5} \lesssim  R^2\|v\|_{X^{\sigma,b}}\|v_4\|_{X^{0,1-b-\theta}}
\end{align*}
outside a set of probability
\[\leq C\exp\big(-c\frac {R^2}{T^{1-\ep}\|\phi\|_{H^s}}\big).\]
\noindent
6th Term : $zzv$ term
\begin{align}\label{term6}
\Big|\int\int_{\mathbb{R} \times \mathbb{R}^d} \langle \nabla \rangle^\sigma (|x|^{-2\al}*z_1z_2)v_3v_4 dxdt\Big|.
\end{align}

We consider 3 cases separately
\begin{enumerate}[\quad i.]
\item $N_3 \gtrsim N_2$
\item $N_3 \ll N_2 \sim N_1$
\item $N_3, N_1 \ll N_2$.
\end{enumerate}

Case (6.i) : $N_3 \gtrsim N_2$.

H\"older inequality and Lemma \ref{mass-xsb} yield that $\eqref{term6}$ is bounded by
\begin{align*}
\|\langle \nabla \rangle^\sigma ((|x|^{-2\al}*(z_1z_2)v_3)\|_{L^{\frac 1{1-\ep}}_tL^{2}_x}\|v_4\|_{L^{\frac {1}{\ep}}_tL^2_x} \lesssim \|\langle \nabla \rangle^\sigma ((|x|^{-2\al}*(z_1z_2)v_3)\|_{L^{\frac 1{1-\ep}}_tL^{2}_x}\|v_4\|_{X^{0,1-b-\theta}},
\end{align*}
for some small $\ep$ such that $0< \ep < \frac {1}{\al}(\sigma - \frac {\al}{2})$.
Then we use Lemma \ref{lei} and Lemma \ref{str-xsb} to get
\begin{align*}
&\qquad \|\langle \nabla \rangle^\sigma ((|x|^{-2\al}*(|\eta_Tv|^2)\eta_Tv)\|_{L^{\frac 1{1-\ep}}_tL^{2}_x}\\
&\lesssim \||x|^{-2\al}*(z_1z_2)\|_{L^{\frac {1}{1-\ep}}_tL^{\infty}_x}\|\langle \nabla \rangle^\sigma v_3\|_{L^\infty_tL^2_x} + \|\langle \nabla \rangle^\sigma (|x|^{-2\al}*(z_1z_2))\|_{L^{\frac {1}{1-\ep}}_tL^{\infty}_x}\|v_3\|_{L^\infty_tL^2_x}\\
&\lesssim \||x|^{-2\al}*(z_1z_2)\|_{L^{\frac {1}{1-\ep}}_tL^{\infty}_x}\|v_3\|_{X^{\sigma, b}} + \|\langle \nabla \rangle^\sigma (|x|^{-2\al}*(z_1z_2))\|_{L^{\frac {1}{1-\ep}}_tL^{\infty}_x}\|v_3\|_{X^{0, b}}.
\end{align*}
Thereafter, from Lemma \ref{frac bound} and H\"older inequality, we obtain
\begin{align*}
&\qquad \||x|^{-2\al}*(z_1z_2)\|_{L^{\frac {1}{1-\ep}}_tL^{\infty}_x} \lesssim \|z_1\|_{L^{\frac {2}{1-\ep}}_tL^{\frac {2d}{d-2\al-\al\ep}}_x}\|z_2\|_{L^{\frac {2}{1-\ep}}_tL^{\frac {2d}{d-2\al + \al\ep}}_x}
\end{align*}
and
\begin{align*}
&\qquad \|\langle \nabla \rangle^\sigma(|x|^{-2\al}*(z_1z_2))\|_{L^{\frac {1}{1-\ep}}_tL^{\infty}_x} = \||x|^{-2\al}*(\langle \nabla \rangle^\sigma(z_1z_2))\|_{L^{\frac {1}{1-\ep}}_tL^{\infty}_x}\\
&\lesssim \|\langle \nabla \rangle^\sigma(z_1z_2)\|^{\frac12}_{L^{\frac {1}{1-\ep}}_tL^{\frac {d}{d-2\al-\al\ep}}_x}\|\langle \nabla \rangle^\sigma(z_1z_2)\|^{\frac12}_{L^{\frac {1}{1-\ep}}_tL^{\frac {d}{d-2\al + \al\ep}}_x}\\
&\lesssim (\|\langle \nabla \rangle^\sigma z_1\|^{\frac12}_{L^{\frac {2}{1-\ep}}_tL^{\frac {2d}{d-2\al-\al\ep}}_x}\|z_2\|^{\frac12}_{L^{\frac {2}{1-\ep}}_tL^{\frac {2d}{d-2\al - \al\ep}}_x} + \|z_1\|^{\frac12}_{L^{\frac {2}{1-\ep}}_tL^{\frac {2d}{d-2\al-\al\ep}}_x}\|\langle \nabla \rangle^\sigma z_2\|^{\frac12}_{L^{\frac {2}{1-\ep}}_tL^{\frac {2d}{d-2\al - \al\ep}}_x})\\
&\times (\|\langle \nabla \rangle^\sigma z_1\|^{\frac12}_{L^{\frac {2}{1-\ep}}_tL^{\frac {2d}{d-2\al+\al\ep}}_x}\|z_2\|^{\frac12}_{L^{\frac {2}{1-\ep}}_tL^{\frac {2d}{d-2\al + \al\ep}}_x} + \|z_1\|^{\frac12}_{L^{\frac {2}{1-\ep}}_tL^{\frac {2d}{d-2\al+\al\ep}}_x}\|\langle \nabla \rangle^\sigma z_2\|^{\frac12}_{L^{\frac {2}{1-\ep}}_tL^{\frac {2d}{d-2\al + \al\ep}}_x}).
\end{align*}
Now we use Bernstein inequality and $N_3 \gtrsim N_2$ to get
\begin{align*}
&\quad \|\langle \nabla \rangle^\sigma(|x|^{-2\al}*(z_1z_2))\|_{L^{\frac {1}{1-\ep}}_tL^{\infty}_x} \\
&\lesssim N_3^\sigma \|z_1\|^{\frac12}_{L^{\frac {2}{1-\ep}}_tL^{\frac {2d}{d-2\al-\al\ep}}_x}\|z_2\|^{\frac12}_{L^{\frac {2}{1-\ep}}_tL^{\frac {2d}{d-2\al - \al\ep}}_x} \|z_1\|^{\frac12}_{L^{\frac {2}{1-\ep}}_tL^{\frac {2d}{d-2\al+\al\ep}}_x}\|z_2\|^{\frac12}_{L^{\frac {2}{1-\ep}}_tL^{\frac {2d}{d-2\al + \al\ep}}_x}.
\end{align*}

Therefore, from Lemma \ref{prbStr-1}, we conclude that
\begin{align*}
\sum_{N_3 \gtrsim N_2}\eqref{term6} &\lesssim R^2\|v\|_{X^{\sigma,b}}\|v_4\|_{X^{0,1-b-\theta}}
\end{align*}
outside a set of probability
\[\leq C\exp\big(-c\frac {R^2}{T^{1-\ep}\|\phi\|_{H^{0+}}}\big).\]

Case (6.ii) : $N_3 \ll N_2 \sim  N_1$

We assume $N_4 \geq N_3$, because other case can be handle similarly.

From H\"older inequality and Lemma \ref{mass-xsb}, we get
\begin{align*}
\eqref{term6} &\lesssim \|\langle \nabla \rangle^\sigma ((|x|^{-2\al}*(z_1z_2)v_3)\|_{L^{\frac 1{1-\ep}}_tL^{2}_x}\|v_4\|_{L^{\frac {1}{\ep}}_tL^2_x}\\
&\lesssim \|\langle \nabla \rangle^\sigma ((|x|^{-2\al}*(z_1z_2)v_3)\|_{L^{\frac 1{1-\ep}}_tL^{2}_x}\|v_4\|_{X^{0,1-b-\theta}},
\end{align*}
for some small positive $\ep$ such that $\ep < \frac {1}{\al}(\sigma - \frac {\al}{2})$.
Then Lemma \ref{lei} and Lemma \ref{str-xsb} give
\begin{align*}
&\qquad \|\langle \nabla \rangle^\sigma ((|x|^{-2\al}*(|\eta_Tv|^2)\eta_Tv)\|_{L^{\frac 1{1-\ep}}_tL^{2}_x}\\
&\lesssim \||x|^{-2\al}*(z_1z_2)\|_{L^{\frac {1}{1-\ep}}_tL^{\infty}_x}\|\langle \nabla \rangle^\sigma v_3\|_{L^\infty_tL^2_x} + \|\langle \nabla \rangle^\sigma (|x|^{-2\al}*(z_1z_2))\|_{L^{\frac {1}{1-\ep}}_tL^{\infty}_x}\|v_3\|_{L^\infty_tL^2_x}\\
&\lesssim \||x|^{-2\al}*(z_1z_2)\|_{L^{\frac {1}{1-\ep}}_tL^{\infty}_x}\|v_3\|_{X^{\sigma, b}} + \|\langle \nabla \rangle^\sigma (|x|^{-2\al}*(z_1z_2))\|_{L^{\frac {1}{1-\ep}}_tL^{\infty}_x}\|v_3\|_{X^{0, b}}.
\end{align*}
Thereafter we use Lemma \ref{frac bound} and H\"older inequality to get
\begin{align*}
&\qquad \||x|^{-2\al}*(z_1z_2)\|_{L^{\frac {1}{1-\ep}}_tL^{\infty}_x} \lesssim \|z_1\|_{L^{\frac {2}{1-\ep}}_tL^{\frac {2d}{d-2\al-\al\ep}}_x}\|z_2\|_{L^{\frac {2}{1-\ep}}_tL^{\frac {2d}{d-2\al + \al\ep}}_x}
\end{align*}
and
\begin{align*}
&\qquad \|\langle \nabla \rangle^\sigma(|x|^{-2\al}*(z_1z_2))\|_{L^{\frac {1}{1-\ep}}_tL^{\infty}_x} = \||x|^{-2\al}*(\langle \nabla \rangle^\sigma(z_1z_2))\|_{L^{\frac {1}{1-\ep}}_tL^{\infty}_x}\\
&\lesssim \|\langle \nabla \rangle^\sigma(z_1z_2)\|^{\frac12}_{L^{\frac {1}{1-\ep}}_tL^{\frac {d}{d-2\al-\al\ep}}_x}\|\langle \nabla \rangle^\sigma(z_1z_2)\|^{\frac12}_{L^{\frac {1}{1-\ep}}_tL^{\frac {d}{d-2\al + \al\ep}}_x}\\
&\lesssim (\|\langle \nabla \rangle^\sigma z_1\|^{\frac12}_{L^{\frac {2}{1-\ep}}_tL^{\frac {2d}{d-2\al-\al\ep}}_x}\|z_2\|^{\frac12}_{L^{\frac {2}{1-\ep}}_tL^{\frac {2d}{d-2\al - \al\ep}}_x} + \|z_1\|^{\frac12}_{L^{\frac {2}{1-\ep}}_tL^{\frac {2d}{d-2\al-\al\ep}}_x}\|\langle \nabla \rangle^\sigma z_2\|^{\frac12}_{L^{\frac {2}{1-\ep}}_tL^{\frac {2d}{d-2\al - \al\ep}}_x})\\
&\times (\|\langle \nabla \rangle^\sigma z_1\|^{\frac12}_{L^{\frac {2}{1-\ep}}_tL^{\frac {2d}{d-2\al+\al\ep}}_x}\|z_2\|^{\frac12}_{L^{\frac {2}{1-\ep}}_tL^{\frac {2d}{d-2\al + \al\ep}}_x} + \|z_1\|^{\frac12}_{L^{\frac {2}{1-\ep}}_tL^{\frac {2d}{d-2\al+\al\ep}}_x}\|\langle \nabla \rangle^\sigma z_2\|^{\frac12}_{L^{\frac {2}{1-\ep}}_tL^{\frac {2d}{d-2\al + \al\ep}}_x}).
\end{align*}
Then by using Bernstein inequality and $N_1 \sim N_2$, we have
\begin{align*}
\|\langle \nabla \rangle^\sigma(|x|^{-2\al}*(z_1z_2))\|_{L^{\frac {1}{1-\ep}}_tL^{\infty}_x} \lesssim& N_1^{\sigma -2s}
\prod_{j=1}^2\|\langle \nabla \rangle^s z_j\|^{\frac12}_{L^{\frac {2}{1-\ep}}_tL^{\frac {2d}{d-2\al-\al\ep}}_x}\|\langle \nabla \rangle^s z_j\|^{\frac12}_{L^{\frac {2}{1-\ep}}_tL^{\frac {2d}{d-2\al+\al\ep}}_x}.
\end{align*}

Therefore, from Lemma \ref{prbStr-1}, we conclude that
\begin{align*}
\sum_{N_3 \ll N_2 \sim  N_1}\eqref{term6} &\lesssim R^2\|v\|_{X^{\sigma,b}}\|v_4\|_{X^{0,1-b-\theta}}
\end{align*}
outside a set of probability
\[\leq C\exp\big(-c\frac {R^2}{T^{1-\ep}\|\phi\|_{H^s}}\big).\]

Case (6.iii) : $N_3, N_1 \ll N_2$

We consider 4 cases separately
\begin{enumerate}[\qquad a.]
\item $N_4 \sim N_2 \gg N_2^{\frac {\al - 1}{2\al -1}} \gg N_3, N_1$
\item $N_4 \sim N_2 \gg N_3 \gg N_2^{\frac {\al - 1}{2\al -1}} \gg N_1$
\item $N_4 \sim N_2 \gg N_1 \gg N_2^{\frac {\al - 1}{2\al -1}} \gg N_3$
\item $N_4 \sim N_2 \gg N_3, N_1 \gg N_2^{\frac {\al - 1}{2\al -1}}$.
\end{enumerate}

Subcase (6.iii.a) : $N_4 \sim N_2 \gg N_2^{\frac {\al - 1}{2\al -1}} \gg N_3, N_1$

Similarly to Case (2.ii.a), we need to estimate
\begin{align*}
&\sum_{k,l \in \mathbb{Z}^d} C_{k,l} N_2^{2\al - d} \int\int_{\mathbb{R} \times \mathbb{R}^d} \Big|\langle \nabla \rangle^\sigma  (z^k_1z^l_2v_3)v_4\Big|dxdt.
\end{align*}
So we consider
\begin{align*}
N_2^{2\al - d} \int\int_{\mathbb{R} \times \mathbb{R}^d} &\big|  \langle \nabla \rangle^\sigma z^k_{1}z^l_{2}v_3v_4\big|dxdt,\ N_2^{2\al - d} \int\int_{\mathbb{R} \times \mathbb{R}^d} \big|z^k_{1}\langle \nabla \rangle^\sigma z^l_{2}v_3v_4\big|dxdt\\
&\text{ and}N_2^{2\al - d} \int\int_{\mathbb{R} \times \mathbb{R}^d} \big|z^k_{1}z^l_{2}\langle \nabla \rangle^\sigma v_3v_4\big|dxdt.
\end{align*}
Second term will be only considered, because remaining two terms can be handled similarly. By using H\"older inequality, Lemma \ref{xsbbi-1} and Lemma \ref{xsbbi-2}, we have
\begin{align*}
&\quad N_2^{2\al - d} \int\int_{\mathbb{R} \times \mathbb{R}^d} \big|z^k_{1}\langle \nabla \rangle^\sigma z^l_{2}v_3v_4\big|dxdt \lesssim N_2^{2\al - d}\|\langle \nabla \rangle^\sigma z^l_2 v_3\|_{L^2_{t,x}}\|z^k_1v_4\|_{L^2_{t,x}}\\
&\lesssim N_3^{\frac {d-1}{2}} N_2^{\frac {1-\al}{2} + 2\al - d}N_1^{\frac {d-1}{2} + 2\ep}N_4^{\frac {1-\al}{2} + 2\ep}\|\langle \nabla \rangle^\sigma z^l_{2}\|_{X^{0,b}}\|v_3\|_{X^{0,b}}\|z^k_{1}\|_{X^{0,b}}\|v_4\|_{X^{0,\frac12 - \ep}}
\end{align*}
Thereafter, from $\sum_{k,l}|C_{k,l}| \leq C$ (with $\|\langle \nabla \rangle^\sigma z^l_{2}\|_{X^{0,b}} = \|\langle \nabla \rangle^\sigma z_{2}\|_{X^{0,b}}, \|z^k_{1}\|_{X^{0,b}} = \|z_{1}\|_{X^{0,b}}$), Bernstein inequality and Lemma \ref{xsb}, we obtain
\begin{align*}
&\quad \sum_{k,l \in \mathbb{Z}^d}C_{k,l} N_2^{2\al - d} \int\int_{\mathbb{R} \times \mathbb{R}^d} \big|z^k_{1}\langle \nabla \rangle^\sigma z^l_{2}v_3v_4\big|dxdt\\
&\lesssim N_3^{\frac {d-1}{2}} N_2^{\frac {1-\al}{2} + 2\al - d}N_1^{\frac {d-1}{2} + 2\ep}N_4^{\frac {1-\al}{2} + 2\ep}\|\langle \nabla \rangle^\sigma z_{2}\|_{X^{0,b}}\|v_3\|_{X^{0,b}}\|z_{1}\|_{X^{0,b}}\|v_4\|_{X^{0,\frac12 - \ep}}\\
&\lesssim T^{0-} N_3^{\frac {d-1}{2} - \sigma} N_2^{\frac {1-\al}{2} + 2\al - d + \sigma - s}N_1^{\frac {d-1}{2} - s + 2\ep}N_4^{\frac {1-\al}{2} + 2\ep}\prod^2_{j=1}\|P_{N_j}\phi^\omega\|_{H^s}\|v_3\|_{X^{\sigma,b}}\|v_4\|_{X^{0,\frac12 - \ep}}
\end{align*}

Then we carry out summation in $N_1$ and $N_3$ so that
\begin{align*}
&\quad \sum_{N_3=1}^{N_2^{\frac {\al - 1}{2\al -1}}}\sum_{N_1=1}^{N_2^{\frac {\al - 1}{2\al -1}}}N_2^{2\al - d} \int\int_{\mathbb{R} \times \mathbb{R}^d} \big|z^k_{1}\langle \nabla \rangle^\sigma z^l_{2}v_3v_4\big|dxdt\\
&\lesssim T^{0-} N_2^{\al\frac{2\al-d}{2\al-1} + \sigma(\frac {\al}{2\al - 1}) - s(\frac{3\al-2}{2\al-1}) + 2\ep(\frac{3\al-2}{2\al-1})} \|\phi^\omega\|_{H^s}\|P_{N_3}\phi^\omega\|_{H^s}\|v\|_{X^{\sigma,b}}\|v_4\|_{X^{0,\frac12 - \ep}}.
\end{align*}
Since $s > \sigma\frac{2\al-1}{4\al -3}>\sigma\frac{\al}{3\al-2}$, summation in $N_2$ can be also carried out.

Hence, from Lemma \ref{prbest-2}, we have
\[\sum_{N_4 \sim N_2 \gg N_2^{\frac {\al - 1}{2\al -1}} \gg N_3, N_1}\eqref{term6} \lesssim T^{0-}R^2\|v\|_{X^{\sigma,b}}\|v_4\|_{X^{0,\frac12 - \ep}}\] outside a set of probability \[C\exp(-c\frac{R^2}{\|\phi\|^2_{H^s}}).\]

Subcase (6.iii.b) : $N_4 \sim N_2 \gg N_3 \gg N_2^{\frac {\al - 1}{2\al -1}} \gg N_1$

Similarly to Case (2.ii.a), we need to deal with
\begin{align*}
&\sum_{k,l \in \mathbb{Z}^d} C_{k,l} N_2^{2\al - d} \int\int_{\mathbb{R} \times \mathbb{R}^d} \Big|\langle \nabla \rangle^\sigma  (z^k_1z^l_2v_3)v_4\Big|dxdt.
\end{align*}
So we have to estimate
\begin{align*}
N_2^{2\al - d} \int\int_{\mathbb{R} \times \mathbb{R}^d} &\big|  \langle \nabla \rangle^\sigma z^k_{1}z^l_{2}v_3v_4\big|dxdt,\ N_2^{2\al - d} \int\int_{\mathbb{R} \times \mathbb{R}^d} \big|z^k_{1}\langle \nabla \rangle^\sigma z^l_{2}v_3v_4\big|dxdt\\
&\text{and }N_2^{2\al - d} \int\int_{\mathbb{R} \times \mathbb{R}^d} \big|z^k_{1}z^l_{2}\langle \nabla \rangle^\sigma v_3v_4\big|dxdt.
\end{align*}
Second term will be only considered, because remaining two terms can be handled similarly. H\"older inequality and Lemma \ref{xsbbi-2} yield
\begin{align*}
&N_2^{2\al - d} \int\int_{\mathbb{R} \times \mathbb{R}^d} \big|z^k_{1}z^l_{2}\langle \nabla \rangle^\sigma v_3v_4\big|dxdt \lesssim N_2^{2\al - d} \|v_3\|_{L^q_tL^r_x}\|\langle \nabla \rangle^\sigma z^l_2\|_{L^{\frac{2q}{q-2}}_tL^{\frac{2r}{r-2}}_x}\|z^k_1v_4\|_{L^2_{t,x}}\\
&\lesssim N_2^{2\al - d} N_1^{\frac {d-1}{2} + 2\ep}N_4^{\frac {1-\al}{2} + 2\ep} \|v_3\|_{L^q_tL^r_x}\|\langle \nabla \rangle^\sigma z^l_2\|_{L^{\frac{2q}{q-2}}_tL^{\frac{2r}{r-2}}_x}\|z^k_1\|_{X^{0,b}}\|v_4\|_{X^{0,\frac12 - \ep}}
\end{align*}
Thereafter we use $\sum_{k,l}|C_{k,l}|\leq C$ (with $\|z^k_1\|_{X^{0,b}}=\|z_1\|_{X^{0,b}}, \|\langle \nabla \rangle^\sigma z^l_2\|_{L^{\frac{2q}{q-2}}_tL^{\frac{2r}{r-2}}_x} = \\\|\langle \nabla \rangle^\sigma z_2\|_{L^{\frac{2q}{q-2}}_tL^{\frac{2r}{r-2}}_x}$), Bernstein inequality and Lemma \ref{str-xsb} to obtain
\begin{align*}
&\quad \sum_{k,l \in \mathbb{Z}^d}C_{k,l} N_2^{2\al - d} \int\int_{\mathbb{R} \times \mathbb{R}^d} \big|z^k_{1}z^l_{2}\langle \nabla \rangle^\sigma v_3v_4\big|dxdt\\
&\lesssim N_3^{\frac{2-\al}{q}-\sigma}N_2^{2\al - d + \sigma - s}N_1^{\frac {d-1}{2} - s + 2\ep}N_4^{\frac {1-\al}{2} + 2\ep}\|v_3\|_{X^{\sigma,b}}\|z_1\|_{X^{s,b}}\|\langle \nabla \rangle^s z_2\|_{L^{\frac{2q}{q-2}}_tL^{\frac{2r}{r-2}}_x}\|v_4\|_{X^{0,\frac12 - \ep}},
\end{align*}
where $q = \max(3, \frac{3(\al-1)(2-\al)}{(3\al-1)(d-2\al)})$.

After that, we carry out summation in $N_1$ and $N_3$ and apply Lemma \ref{xsb} as follows
\begin{align*}
&\quad \sum_{N_3=N_2^{\frac {\al - 1}{2\al -1}}}^{N_2}\sum_{N_1=1}^{N_2^{\frac {\al - 1}{2\al -1}}} \sum_{k,l \in \mathbb{Z}^d}C_{k,l} N_2^{2\al - d} \int\int_{\mathbb{R} \times \mathbb{R}^d} \big|z^k_{1}z^l_{2}\langle \nabla \rangle^\sigma v_3v_4\big|dxdt\\
&\lesssim N_2^{\frac{\al-1}{2\al-1}(\frac{2-\al}{q}+\frac{d-2\al}{2(\al-1)}(1-3\al)) + \frac{\al}{2\al-1}\sigma - \frac{3\al-2}{2\al-1}s + 2\frac{3\al-2}{2\al-1}\ep}\|v\|_{X^{\sigma,b}}\|z\|_{X^{s,b}}\|\langle \nabla \rangle^s z_2\|_{L^{\frac{2q}{q-2}}_tL^{\frac{2r}{r-2}}_x}\|v_4\|_{X^{0,\frac12 - \ep}}\\
&\lesssim T^{0-}N_2^{\frac{\al-1}{2\al-1}(\frac{2-\al}{q}+\frac{d-2\al}{2(\al-1)}(1-3\al)) + \frac{\al}{2\al-1}\sigma - \frac{3\al-2}{2\al-1}s + 2\frac{3\al-2}{2\al-1}\ep}\|v\|_{X^{\sigma,b}}\|\phi^\omega\|_{H^s}\|\langle \nabla \rangle^s z_2\|_{L^{\frac{2q}{q-2}}_tL^{\frac{2r}{r-2}}_x}\|v_4\|_{X^{0,\frac12 - \ep}}
\end{align*}
Since the exponent of $N_2$ is negative, summation in $N_2$ can be also carried out.

Therefore, from Lemma \ref{prbest-2} and Lemma \ref{prbStr-1}, we conclude that
\[\sum_{N_4 \sim N_2 \gg N_3 \gg N_2^{\frac {\al - 1}{2\al -1}} \gg N_1}\eqref{term6} \lesssim T^{0-}R^2\|v\|_{X^{\sigma,b}}\|v_4\|_{X^{0,\frac12 -2\ep}}\] outside a set of probability \[C\exp(-c\frac{R^2}{\|\phi\|^2_{H^s}}) + C\exp(-c\frac{R^2}{T^{\frac{q-2}{q}}\|\phi\|^2_{H^s}}).\]

Case(6.iii.c) : $N_4 \sim N_2 \gg N_1 \gg N_2^{\frac {\al - 1}{2\al -1}} \gg N_3$

Similarly to Case (2.ii.a), we estimate the following :
\begin{align*}
&\sum_{k,l \in \mathbb{Z}^d} C_{k,l} N_2^{2\al - d} \int\int_{\mathbb{R} \times \mathbb{R}^d} \Big|\langle \nabla \rangle^\sigma  (z^k_1z^l_2v_3)v_4\Big|dxdt.
\end{align*}
So we need to handle
\begin{align*}
N_2^{2\al - d} \int\int_{\mathbb{R} \times \mathbb{R}^d} &\big|\langle \nabla \rangle^\sigma z^k_{1}z^l_{2}v_3v_4\big|dxdt,\ N_2^{2\al - d} \int\int_{\mathbb{R} \times \mathbb{R}^d} \big|z^k_{1}\langle \nabla \rangle^\sigma z^l_{2}v_3v_4\big|dxdt\\
&\text{and }N_2^{2\al - d} \int\int_{\mathbb{R} \times \mathbb{R}^d} \big|z^k_{1}z^l_{2}\langle \nabla \rangle^\sigma v_3v_4\big|dxdt.
\end{align*}
Second term will be only considered, bewcause remaining two terms can be handled similarly. H\"older inequality and Lemma \ref{xsbbi-2} give
\begin{align*}
&\quad N_2^{2\al - d} \int\int_{\mathbb{R} \times \mathbb{R}^d} \big|z^k_{1}\langle \nabla \rangle^\sigma z^l_{2}v_3v_4\big|dxdt \lesssim N_2^{2\al - d} \|z^k_1\|_{L^4_{t,x}}\|\langle \nabla \rangle^\sigma z^l_2\|_{L^4_{t,x}}\|v_3v_4\|_{L^2_{t,x}}\\
&\lesssim N_2^{2\al - d}N_3^{\frac{d-1}{2} + 2\ep}N_4^{\frac {1-\al}{2} + 2\ep}\|z^k_1\|_{L^4_{t,x}}\|\langle \nabla \rangle^\sigma z^l_2\|_{L^4_{t,x}}\|v_3\|_{X^{0,b}}\|v_4\|_{X^{0,\frac12 - \ep}}
\end{align*}
Then, from $\sum_{k,l}|C_{k,l}| \leq C$ (with $\|z^k_1\|_{L^4_{t,x}}=\|z_1\|_{L^4_{t,x}},\|\langle \nabla \rangle^\sigma z^l_2\|_{L^4_{t,x}}=\|\langle \nabla \rangle^\sigma z_2\|_{L^4_{t,x}}$) and Bernstein inequality, we obtain
\begin{align*}
&\quad \sum_{k,l \in \mathbb{Z}^d}C_{k,l} N_2^{2\al - d} \int\int_{\mathbb{R} \times \mathbb{R}^d} \big|z^k_{1}\langle \nabla \rangle^\sigma z^l_{2}v_3v_4\big|dxdt\\
&\lesssim N_1^{-s}N_2^{2\al - d +\sigma - s}N_3^{\frac{d-1}{2} - \sigma + 2\ep}N_4^{\frac {1-\al}{2} + 2\ep}\prod^2_{j=1}\|\langle \nabla \rangle^s z_j\|_{L^4_{t,x}}\|v_3\|_{X^{\sigma,b}}\|v_4\|_{X^{0,\frac12 - \ep}}.
\end{align*}

Since $N_4 \sim N_2 \gg N_1 \gg N_2^{\frac {\al - 1}{2\al -1}} \gg N_3$, $\frac {d-1}{2} - \sigma + 2\ep>0$ and $-s < 0$, we can carry out summation in $N_1$ and $N_3$ such that
\begin{align*}
&\quad \sum_{N_1=N_2^{\frac {\al - 1}{2\al -1}}}^{N_2} \sum_{N_2=1}^{N_2^{\frac {\al - 1}{2\al -1}}}\sum_{k,l \in \mathbb{Z}^d}C_{k,l} N_2^{2\al - d} \int\int_{\mathbb{R} \times \mathbb{R}^d} \big|z^k_{1}\langle \nabla \rangle^\sigma z^l_{2}v_3v_4\big|dxdt\\
&\lesssim N_2^{\frac{(d-2\al)(1-3\al)}{2(2\al-1)} + \frac{\al}{2\al-1}\sigma - \frac{3\al-2}{2\al-1}s + \frac{3\al-2}{2\al-1}\ep_1}\|\langle \nabla \rangle^s z\|_{L^4_{t,x}}\|\langle \nabla \rangle^s z_2\|_{L^4_{t,x}}\|v\|_{X^{\sigma,b}}\|v_4\|_{X^{0,\frac12 -2\ep}}
\end{align*}
After carrying out summation in $N_2$, from Lemma \ref{prbStr-1}, we conclude that
\[\sum_{N_4 \sim N_2 \gg N_1 \gg N_2^{\frac {\al - 1}{2\al -1}} \gg N_3} \eqref{term6} \lesssim R^2\|v\|_{X^{\sigma,b}}\|v_4\|_{X^{0,\frac12 -2\ep}}\] outside a set of probability \[C\exp(-c\frac{R^2}{T^{\frac12}\|\phi\|^2_{H^s}}).\]

Subcase(6.iii.d) : $N_4 \sim N_2 \gg N_3, N_1 \gg N_2^{\frac {\al - 1}{2\al -1}}$

H\"older inequality and Lemma \ref{mass-xsb} yield $\eqref{term6}$ is bounded by
\begin{align*}
\|\langle \nabla \rangle^\sigma ((|x|^{-2\al}*(z_1z_2)v_3)\|_{L^{\frac 1{1-\ep}}_tL^{2}_x}\|v_4\|_{L^{\frac {1}{\ep}}_tL^2_x} \lesssim \|\langle \nabla \rangle^\sigma ((|x|^{-2\al}*(z_1z_2)v_3)\|_{L^{\frac 1{1-\ep}}_tL^{2}_x}\|v_4\|_{X^{0,1-b-\theta}},
\end{align*}
for some small $\ep$ such that $0< \ep < \frac {1}{\al}(\sigma - \frac {\al}{2})$.
Then we use Lemma \ref{lei} and Lemma \ref{str-xsb} to obtain
\begin{align*}
&\qquad \|\langle \nabla \rangle^\sigma ((|x|^{-2\al}*(z_1z_2)v_3)\|_{L^{\frac 1{1-\ep}}_tL^{2}_x}\\
&\lesssim \||x|^{-2\al}*(z_1z_2)\|_{L^{\frac {1}{1-\ep}}_tL^{\infty}_x}\|\langle \nabla \rangle^\sigma v_3\|_{L^\infty_tL^2_x} + \|\langle \nabla \rangle^\sigma (|x|^{-2\al}*(z_1z_2))\|_{L^{\frac {1}{1-\ep}}_tL^{\infty}_x}\|v_3\|_{L^\infty_tL^2_x}\\
&\lesssim \||x|^{-2\al}*(z_1z_2)\|_{L^{\frac {1}{1-\ep}}_tL^{\infty}_x}\|v_3\|_{X^{\sigma, b}} + \|\langle \nabla \rangle^\sigma (|x|^{-2\al}*(z_1z_2))\|_{L^{\frac {1}{1-\ep}}_tL^{\infty}_x}\|v_3\|_{X^{0, b}}.
\end{align*}
Thereafter, from Lemma \ref{frac bound} and H\"older inequality, we obtain
\begin{align*}
&\qquad \||x|^{-2\al}*(z_1z_2)\|_{L^{\frac {1}{1-\ep}}_tL^{\infty}_x} \lesssim \|z_1\|_{L^{\frac {2}{1-\ep}}_tL^{\frac {2d}{d-2\al-\al\ep}}_x}\|z_2\|_{L^{\frac {2}{1-\ep}}_tL^{\frac {2d}{d-2\al + \al\ep}}_x}\\
&\qquad \lesssim N_1^{-s}N_2^{-s}\|\langle \nabla \rangle^s z_1\|_{L^{\frac {2}{1-\ep}}_tL^{\frac {2d}{d-2\al-\al\ep}}_x}\|\langle \nabla \rangle^s z_2\|_{L^{\frac {2}{1-\ep}}_tL^{\frac {2d}{d-2\al + \al\ep}}_x}
\end{align*}
and
\begin{align*}
&\qquad \|\langle \nabla \rangle^\sigma(|x|^{-2\al}*(z_1z_2))\|_{L^{\frac {1}{1-\ep}}_tL^{\infty}_x} = \||x|^{-2\al}*(\langle \nabla \rangle^\sigma(z_1z_2))\|_{L^{\frac {1}{1-\ep}}_tL^{\infty}_x}\\
&\lesssim \|\langle \nabla \rangle^\sigma(|x|^{-2\al}*(z_1z_2))\|^{\frac12}_{L^{\frac {1}{1-\ep}}_tL^{\frac {d}{d-2\al-\al\ep}}_x}\|\langle \nabla \rangle^\sigma(|x|^{-2\al}*(z_1z_2))\|^{\frac12}_{L^{\frac {1}{1-\ep}}_tL^{\frac {d}{d-2\al + \al\ep}}_x}\\
&\lesssim (\|\langle \nabla \rangle^\sigma z_1\|^{\frac12}_{L^{\frac {2}{1-\ep}}_tL^{\frac {2d}{d-2\al-\al\ep}}_x}\|z_2\|^{\frac12}_{L^{\frac {2}{1-\ep}}_tL^{\frac {2d}{d-2\al - \al\ep}}_x} + \|z_1\|^{\frac12}_{L^{\frac {2}{1-\ep}}_tL^{\frac {2d}{d-2\al-\al\ep}}_x}\|\langle \nabla \rangle^\sigma z_2\|^{\frac12}_{L^{\frac {2}{1-\ep}}_tL^{\frac {2d}{d-2\al - \al\ep}}_x})\\
&\times (\|\langle \nabla \rangle^\sigma z_1\|^{\frac12}_{L^{\frac {2}{1-\ep}}_tL^{\frac {2d}{d-2\al+\al\ep}}_x}\|z_2\|^{\frac12}_{L^{\frac {2}{1-\ep}}_tL^{\frac {2d}{d-2\al + \al\ep}}_x} + \|z_1\|^{\frac12}_{L^{\frac {2}{1-\ep}}_tL^{\frac {2d}{d-2\al+\al\ep}}_x}\|\langle \nabla \rangle^\sigma z_2\|^{\frac12}_{L^{\frac {2}{1-\ep}}_tL^{\frac {2d}{d-2\al + \al\ep}}_x}).
\end{align*}
Then we use Bernstein inequality and $N_2 \gtrsim N_1$ to get
\begin{align*}
\|\langle \nabla \rangle^\sigma(|x|^{-2\al}*(z_1z_2))\|_{L^{\frac {1}{1-\ep}}_tL^{\infty}_x} \lesssim& N_2^{\sigma - s}N_1^{-s}
\prod_{j=1}^2\|\langle \nabla \rangle^s z_j\|^{\frac12}_{L^{\frac {2}{1-\ep}}_tL^{\frac {2d}{d-2\al-\al\ep}}_x}\|\langle \nabla \rangle^s z_j\|^{\frac12}_{L^{\frac {2}{1-\ep}}_tL^{\frac {2d}{d-2\al+\al\ep}}_x}.
\end{align*}
Thus, from $N_2 \gtrsim N_3$, we have
\begin{align*}
&\qquad \|\langle \nabla \rangle^\sigma ((|x|^{-2\al}*(z_1z_2)v_3)\|_{L^{\frac 1{1-\ep}}_tL^{2}_x}\\
&\lesssim N_2^{\sigma - s}N_1^{-s}N_3^{-\sigma}\|\langle \nabla \rangle^s z_1\|_{L^{\frac {2}{1-\ep}}_tL^{\frac {2d}{d-2\al-\al\ep}}_x}\|\langle \nabla \rangle^s z_2\|_{L^{\frac {2}{1-\ep}}_tL^{\frac {2d}{d-2\al + \al\ep}}_x}\|v_3\|_{X^{\sigma, b}}\\
& +N_2^{\sigma - s}N_1^{-s}N_3^{-\sigma}(\prod_{j=1}^2\|\langle \nabla \rangle^s z_j\|^{\frac12}_{L^{\frac {2}{1-\ep}}_tL^{\frac {2d}{d-2\al-\al\ep}}_x} \|\langle \nabla \rangle^s z_j\|^{\frac12}_{L^{\frac {2}{1-\ep}}_tL^{\frac {2d}{d-2\al+\al\ep}}_x})\|v_3\|_{X^{\sigma, b}}.
\end{align*}

Now we carry out summation in $N_1$ and $N_3$
\begin{align*}
&\sum_{N_3 = N_2^{\frac {\al - 1}{2\al -1}}}^{N_2}\sum_{N_1 = N_2^{\frac {\al - 1}{2\al -1}}}^{N_2}\|\langle \nabla \rangle^\sigma ((|x|^{-2\al}*(z_1z_2)v_3)\|_{L^{\frac 1{1-\ep}}_tL^{2}_x}\\
&\lesssim N_2^{\frac{\al}{2\al-1}\sigma - \frac{3\al-2}{2\al-1}s}\big(\|\langle \nabla \rangle^s z\|_{L^{\frac {2}{1-\ep}}_tL^{\frac {2d}{d-2\al-\al\ep}}_x}\|\langle \nabla \rangle^s z_2\|_{L^{\frac {2}{1-\ep}}_tL^{\frac {2d}{d-2\al - \al\ep}}_x}\|v\|_{X^{\sigma, b}} +\\
&\|\langle \nabla \rangle^s z\|^{\frac12}_{L^{\frac {2}{1-\ep}}_tL^{\frac {2d}{d-2\al-\al\ep}}_x}\|\langle \nabla \rangle^s z\|^{\frac12}_{L^{\frac {2}{1-\ep}}_tL^{\frac {2d}{d-2\al+\al\ep}}_x}\|\langle \nabla \rangle^s z_2\|^{\frac12}_{L^{\frac {2}{1-\ep}}_tL^{\frac {2d}{d-2\al - \al\ep}}_x}\|\langle \nabla \rangle^s z_2\|^{\frac12}_{L^{\frac {2}{1-\ep}}_tL^{\frac {2d}{d-2\al + \al\ep}}_x}\|v\|_{X^{\sigma, b}}\big)
\end{align*}

Therefore, from Lemma \ref{prbStr-1}, we conclude that
\begin{align*}
\sum_{N_4 \sim N_2 \gg N_3, N_1 \gg N_2^{\frac {\al - 1}{2\al -1}}}\eqref{term6} \lesssim R^2\|v\|_{X^{\sigma,b}}\|v_4\|_{X^{0,1-b-\theta}}
\end{align*}
outside a set of probability
\[\leq C\exp\big(-c\frac {R^2}{T^{1-\ep}\|\phi\|_{H^s}}\big).\]
\end{proof}

\section*{Acknowledgments} The author is grateful to Professors Yonggeun Cho, Soonsik Kwon and Sanghyuk Lee for valuable comments and support.

\bibliographystyle{plain}


\end{document}